\documentclass[10pt]{amsart}
\usepackage{amssymb,amstext,amsmath,amscd,amsthm,amsfonts,enumerate,latexsym,stmaryrd,multicol,geometry,graphicx,mathrsfs}
\usepackage[dvipdfmx]{hyperref}
\usepackage[usenames]{color}
\usepackage[all]{xy}
\newtheorem{thm}{Theorem}[section]
\newtheorem{lem}[thm]{Lemma}

\theoremstyle{definition}
\newtheorem{dfn}[thm]{Definition}
\newtheorem{ques}[thm]{Question}

\newtheorem{rem}[thm]{Remark}
\newtheorem{conv}[thm]{Convention}

\theoremstyle{remark}
\newtheorem{claim}{Claim}
\newtheorem*{claim*}{Claim}

\renewcommand{\qedsymbol}{$\blacksquare$}
\numberwithin{equation}{thm}

\def\a{\mathrm{A}}
\def\add{\operatorname{add}}
\def\C{\mathcal{C}}
\def\cm{\operatorname{CM}}
\def\cok{\operatorname{Cok}}
\def\d{\mathrm{D}}
\def\db{\operatorname{D^b}}
\def\depth{\operatorname{depth}}
\def\ds{\operatorname{D^{sg}}}
\def\dso{\operatorname{D^{sg}_0}}
\def\E{\mathfrak{E}}
\def\e{\mathrm{E}}
\def\edim{\operatorname{edim}}
\def\Ext{\operatorname{Ext}}
\def\ext{\operatorname{ext}}
\def\ge{\geqslant}
\def\gr{\operatorname{gr}}
\def\h{\operatorname{H}}
\def\Hom{\operatorname{Hom}}
\def\ind{\operatorname{ind}}
\def\lcm{\operatorname{\underline{CM}}}
\def\le{\leqslant}
\def\m{\mathfrak{m}}
\def\mod{\operatorname{mod}}
\def\N{\mathbb{N}}
\def\n{\mathfrak{n}}
\def\P{\mathrm{P}}
\def\p{\mathfrak{p}}
\def\q{\mathfrak{q}}
\def\r{\mathfrak{r}}

\def\spec{\operatorname{Spec}}
\def\syz{\Omega}
\def\W{\mathcal{W}}
\def\X{\mathcal{X}}
\def\Y{\mathcal{Y}}
\def\Z{\mathbb{Z}}
\def\zv{\boldsymbol{0}}
\def\ZZ{\mathcal{Z}}
\begin{document}
\title[Extension-closed subcategories and finite/countable CM-representation type]{Extension-closed subcategories over hypersurfaces\\
of finite or countable CM-representation type}
\author{Kei-ichiro Iima}
\address{Department of Liberal Studies, National Institute of Technology, Nara College, 22 Yata-cho, Yamatokoriyama, Nara 639-1080, Japan}
\email{iima@libe.nara-k.ac.jp}
\author{Ryo Takahashi}
\address{Graduate School of Mathematics, Nagoya University, Furocho, Chikusaku, Nagoya 464-8602, Japan}
\email{takahashi@math.nagoya-u.ac.jp}
\urladdr{https://www.math.nagoya-u.ac.jp/~takahashi/}
\thanks{Takahashi was partly supported by JSPS Grant-in-Aid for Scientific Research 23K03070}
\dedicatory{Dedicated to Professor Yuji Yoshino on the occasion of his seventieth birthday}
\begin{abstract}
Let $k$ be an algebraically closed uncountable field of characteristic zero.
Let $R$ be a complete local hypersurface over $k$.
Denote by $\cm(R)$ the category of maximal Cohen--Macaulay $R$-modules and by $\ds(R)$ the singularity category of $R$.
Denote by $\cm_0(R)$ the full category of $\cm(R)$ consisting of modules that are locally free on the punctured spectrum of $R$, and by $\dso(R)$ the full subcategory of $\ds(R)$ consisting of objects that are locally zero on the punctured spectrum of $R$.
In this paper, under the assumption that $R$ has finite or countable CM-representation type, we completely classify the extension-closed subcategories of $\cm_0(R)$ in dimension at most two, and the extension-closed subcategories of $\dso(R)$ in arbitrary dimension.
\end{abstract}
\maketitle
\tableofcontents
\section{Introduction}

Let $R$ be a local hypersurface ring.
By virtue of \cite{stcm,thd} there are one-to-one correspondences between:
\begin{itemize}
\item
the resolving subcategories of $\mod R$ contained in $\cm(R)$,
\item
the thick subcategories of $\ds(R)$,
\item
the specialization-closed subsets of the singular locus of $R$,
\end{itemize}
where $\mod R$, $\cm(R)$, and $\ds(R)$ respectively stand for the category of finitely generated $R$-modules, the full subcategory of maximal Cohen--Macaulay $R$-modules, and the singularity category of $R$, i.e., the Verdier quotient of the bounded derived category of $\mod R$ by perfect complexes.
The bijections are explicitly given, which leads complete classifications of the resolving subcategories and thick subcategories mentioned above.

An {\em extension-closed subcategory} of an extriangulated category $\C$ (in the sense of \cite{NP}) is defined to be a full subcategory $\X$ of $\C$ which is closed under direct summands and such that for each conflation $X\to Y\to Z$ in $\C$, if $X$ and $Z$ belong to $\X$, then $Y$ also belongs to $\X$.
Note by definition that both a resolving subcategory of $\mod R$ and a thick subcategory of $\ds(R)$ are extension-closed subcategories.

Denote by $\cm_0(R)$ the full subcategory of $\cm(R)$ consisting of maximal Cohen--Macaulay modules which are locally free on the punctured spectrum of $R$, and by $\dso(R)$ the full subcategory of $\ds(R)$ consisting of objects which are locally zero on the punctured spectrum of $R$.
Note that the equalities $\cm_0(R)=\cm(R)$ and $\dso(R)=\ds(R)$ hold whenever $R$ has an isolated singularity.
The classification theorem stated above implies that, if $R$ is a local hypersurface ring, then there exist only trivial resolving subcategories of $\mod R$ contained in $\cm_0(R)$ and there exist only trivial thick subcategories of $\ds(R)$ contained in $\dso(R)$.
This fact motivates us to classify the extension-closed subcategories of $\cm_0(R)$ and $\dso(R)$.

In the present paper, we shall prove the following theorem.

\begin{thm}\label{28}
Let $k$ be an algebraically closed uncountable field of characteristic $0$.
Let $R$ be a singular local hypersurface ring with residue field $k$.
Suppose that $R$ has either finite or countable CM-representation type.
\begin{enumerate}[\rm(1)]
\item
The Hasse diagram of the partially ordered set of extension-closed subcategories of $\dso(R)$ with respect to the inclusion relation is one of the following four graphs.
$$
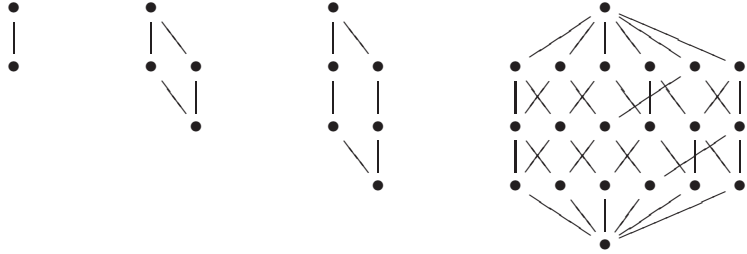
\xymatrix@R-1pc@C-1.5pc{
\bullet\ar@{-}[d]\\
\bullet
}\qquad\qquad
\xymatrix@R-1pc@C-1.5pc{
\bullet\ar@{-}[d]\ar@{-}[rd]\\
\bullet\ar@{-}[rd]&\bullet\ar@{-}[d]\\
&\bullet
}\qquad\qquad
\xymatrix@R-1pc@C-1.5pc{
\bullet\ar@{-}[d]\ar@{-}[rd]\\
\bullet\ar@{-}[d]&\bullet\ar@{-}[d]\\
\bullet\ar@{-}[rd]&\bullet\ar@{-}[d]\\
&\bullet
}\qquad\qquad
\xymatrix@R-1pc@C-1.5pc{
&&\bullet\ar@{-}[lld]\ar@{-}[ld]\ar@{-}[d]\ar@{-}[rd]\ar@{-}[rrd]\ar@{-}[rrrd]\\
\bullet\ar@{-}[d]\ar@{-}[rd]&\bullet\ar@{-}[ld]\ar@{-}[rd]&\bullet\ar@{-}[ld]\ar@{-}[rd]&\bullet\ar@{-}[d]\ar@{-}[rd]&\bullet\ar@{-}[lld]\ar@{-}[rd]&\bullet\ar@{-}[ld]\ar@{-}[d]\\
\bullet\ar@{-}[d]\ar@{-}[rd]&\bullet\ar@{-}[ld]\ar@{-}[rd]&\bullet\ar@{-}[ld]\ar@{-}[rd]&\bullet\ar@{-}[ld]\ar@{-}[rd]&\bullet\ar@{-}[d]\ar@{-}[rd]&\bullet\ar@{-}[lld]\ar@{-}[d]\\
\bullet\ar@{-}[rrd]&\bullet\ar@{-}[rd]&\bullet\ar@{-}[d]&\bullet\ar@{-}[ld]&\bullet\ar@{-}[lld]&\bullet\ar@{-}[llld]\\
&&\bullet
}
$$
\item
If $R$ has Krull dimension at most two, then the Hasse diagram of the partially ordered set of extension-closed subcategories of $\cm_0(R)$ with respect to the inclusion relation is one of the following five graphs.
$$
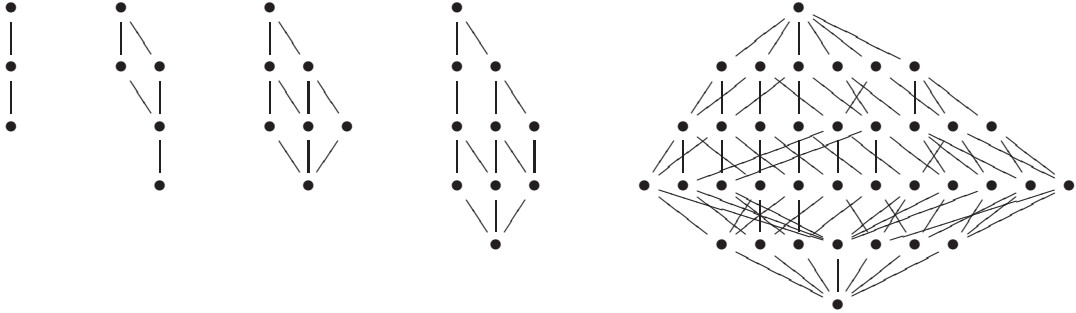
\xymatrix@R-1pc@C-1.7pc{
\bullet\ar@{-}[d]\\
\bullet\ar@{-}[d]\\
\bullet
}\qquad\quad
\xymatrix@R-1pc@C-1.7pc{
\bullet\ar@{-}[d]\ar@{-}[rd]\\
\bullet\ar@{-}[rd]&\bullet\ar@{-}[d]\\
&\bullet\ar@{-}[d]\\
&\bullet
}\qquad\quad
\xymatrix@R-1pc@C-1.7pc{
\bullet\ar@{-}[d]\ar@{-}[rd]\\
\bullet\ar@{-}[d]\ar@{-}[rd]&\bullet\ar@{-}[d]\ar@{-}[rd]\\
\bullet\ar@{-}[rd]&\bullet\ar@{-}[d]&\bullet\ar@{-}[ld]\\
&\bullet
}\qquad\quad
\xymatrix@R-1pc@C-1.7pc{
\bullet\ar@{-}[d]\ar@{-}[rd]\\
\bullet\ar@{-}[d]&\bullet\ar@{-}[d]\ar@{-}[rd]\\
\bullet\ar@{-}[d]\ar@{-}[rd]&\bullet\ar@{-}[d]\ar@{-}[rd]&\bullet\ar@{-}[d]\\
\bullet\ar@{-}[rd]&\bullet\ar@{-}[d]&\bullet\ar@{-}[ld]\\
&\bullet
}\qquad\quad
\xymatrix@R-1pc@C-1.7pc{
&&&&\bullet\ar@{-}[lld]\ar@{-}[ld]\ar@{-}[d]\ar@{-}[rd]\ar@{-}[rrd]\ar@{-}[rrrd]\\
&&\bullet\ar@{-}[ld]\ar@{-}[d]&\bullet\ar@{-}[lld]\ar@{-}[d]\ar@{-}[rrd]&\bullet\ar@{-}[lld]\ar@{-}[d]\ar@{-}[rrd]&\bullet\ar@{-}[rd]\ar@{-}[rrd]&\bullet\ar@{-}[ld]\ar@{-}[rrd]&\bullet\ar@{-}[d]\ar@{-}[rd]\ar@{-}[rrd]\\
&\bullet\ar@{-}[ld]\ar@{-}[d]\ar@{-}[rrd]&\bullet\ar@{-}[lld]\ar@{-}[d]\ar@{-}[rrd]&\bullet\ar@{-}[d]\ar@{-}[rrd]&\bullet\ar@{-}[d]\ar@{-}[rrd]&\bullet\ar@{-}[lllld]\ar@{-}[d]\ar@{-}[rrd]&\bullet\ar@{-}[lllld]\ar@{-}[d]\ar@{-}[rrd]&\bullet\ar@{-}[rd]\ar@{-}[rrd]\ar@{-}[rrrd]&\bullet\ar@{-}[ld]\ar@{-}[rd]\ar@{-}[rrrd]&\bullet\ar@{-}[rd]\ar@{-}[rrd]\\
\bullet\ar@{-}[rrd]\ar@{-}[rrrrrd]&\bullet\ar@{-}[rrd]\ar@{-}[rrrrd]&\bullet\ar@{-}[rrd]\ar@{-}[rrrd]&\bullet\ar@{-}[ld]\ar@{-}[d]&\bullet\ar@{-}[lld]\ar@{-}[d]&\bullet\ar@{-}[lld]\ar@{-}[rd]&\bullet\ar@{-}[lld]\ar@{-}[rd]&\bullet\ar@{-}[lld]\ar@{-}[ld]&\bullet\ar@{-}[llld]\ar@{-}[ld]&\bullet\ar@{-}[lllld]\ar@{-}[ld]&\bullet\ar@{-}[llld]\ar@{-}[lld]&\bullet\ar@{-}[llllld]\ar@{-}[llld]\\
&&\bullet\ar@{-}[rrrd]&\bullet\ar@{-}[rrd]&\bullet\ar@{-}[rd]&\bullet\ar@{-}[d]&\bullet\ar@{-}[ld]&\bullet\ar@{-}[lld]&\bullet\ar@{-}[llld]\\
&&&&&\bullet
}
$$
\end{enumerate}
\end{thm}

\noindent
Note that $0,\add R,\cm_0(R)$ are always extension-closed subcategories of $\cm_0(R)$, and that $0,\dso(R)$ are always extension-closed subcategories of $\dso(R)$.
Hence the first diagrams in the two assertions of the above theorem mean that there exist only trivial extension-closed subcategories.

The organization of this paper is as follows.
In Section 2, we collect definitions and lemmas which are used in later sections.
In Sections 3 and 4, we respectively deal with finite CM-representation type and countable CM-representation type in dimension at most two.
In the final Section 5, we give proofs of Theorem \ref{28}.

\section{Basic definitions and fundamental lemmas}

This section is devoted to stating basic definitions and proving fundamental lemmas for later use.
First of all, let us provide our convention.

\begin{conv}
Let $k$ be an algebraically closed field of characteristic zero, so $\sqrt{-1}\in k$.
We assume that all rings are commutative and noetherian, all modules are finitely generated, and all subcategories are nonempty and strictly full.
Let $R$ be a (commutative noetherian) ring.
We denote by $\mod R$ the category of (finitely generated) $R$-modules, by $\cm(R)$ the (strictly full) subcategory of $\mod R$ consiting of maximal Cohen--Macaulay $R$-modules (we regard the zero $R$-module $0$ as maximal Cohen--Macaulay, so that $0$ belongs to $\cm(R)$), and by $\ds(R)$ the singularity category of $R$, which is defined as the Verdier quotient of the bounded derived category $\db(\mod R)$ of $\mod R$ by perfect complexes.
For each $\P\in\{\a_n,\a_\infty,\d_n,\d_\infty,\e_6,\e_7,\e_8\}$ we write $\P^d$ to indicate that the dimension of the $\P$-singularity considered is $d$.
We may omit a subscript or a superscript when it is clear from the context.
\end{conv}

We recall the notion of an exact square which plays an essential role in the proofs of our main results.

\begin{dfn}
A commutative diagram of homomorphisms of $R$-modules
$$
\xymatrix@R-.5pc@C-.5pc{
A\ar[r]^-a\ar[d]^-c& B\ar[d]^-b\\
C\ar[r]^-d& D
}
$$
is called an {\em exact square} if it is a both pushout and pullback diagram, or in other words, if the sequence $0 \to A\xrightarrow{\binom{a}{c}} B\oplus C\xrightarrow{(-b,d)} D\to0$ is exact.
\end{dfn}

It can be verified directly that the following lemma holds true; see \cite[Lemma 2.2(2)]{trans}.

\begin{lem}\label{5}
Let {\rm(1)} and {\rm(2)} below be exact squares.
Then {\rm(3)} below is an exact square as well.
$$
\xymatrix@R-.5pc@C-.5pc{
A\ar[r]^-a\ar[d]^-d\ar@{}[rd]|{\rm(1)}& B\ar[d]^-g\ar[r]^-b\ar@{}[rd]|{\rm(2)}& C\ar[d]^-c\\
D\ar[r]^-e& E\ar[r]^-f& F
}\qquad\qquad
\xymatrix@R-.5pc@C-.5pc{
A\ar[r]^-{ba}\ar[d]^-d\ar@{}[rd]|{\rm(3)}& C\ar[d]^-c\\
D\ar[r]^-{fe}& F
}
$$
In other words, if the sequences $0 \to A\xrightarrow{\binom{a}{d}} B\oplus D\xrightarrow{(-g,e)} E\to0$ and $0 \to B\xrightarrow{\binom{b}{g}} C\oplus E\xrightarrow{(-c,f)} F\to0$ are exact, then so is the sequence $0 \to A\xrightarrow{\binom{ba}{d}} C\oplus D\xrightarrow{(-c,fe)} F\to0$.
\end{lem}

The following elementary lemma will also be necessary.

\begin{lem}\label{7}
\begin{enumerate}[\rm(1)]
\item
If $0\to A\xrightarrow{\binom{f}{g}}B\oplus C\xrightarrow{(h,l)}D\to0$ and $0\to E\xrightarrow{m}A\xrightarrow{f}B\to0$ are exact sequences in $\mod R$, then there exists an exact sequence $0\to E\xrightarrow{gm}C\xrightarrow{l}D\to0$ in $\mod R$.
\item
If $0\to A\xrightarrow{\binom{h}{l}}B\oplus C\xrightarrow{(f,g)}D\to0$ and $0\to C\xrightarrow{g}D\xrightarrow{m}E\to0$ are exact sequences in $\mod R$, then there exists an exact sequence $0\to A\xrightarrow{h}B\xrightarrow{mf}E\to0$ in $\mod R$.
\end{enumerate}
\end{lem}

\begin{proof}
In the situations of (1) and (2), there respectively exist pullback and pushout diagrams
$$
\xymatrix@R-1pc@C-1pc{
0\ar[r]& E\ar[r]^m\ar@{=}[d]& A\ar[r]^f\ar[d]^g& B\ar[r]\ar[d]^{-h}& 0\\
0\ar[r]& E\ar[r]^{gm}& C\ar[r]^l& D\ar[r]& 0
}\qquad\qquad
\xymatrix@R-1pc@C-1pc{
0\ar[r]& A\ar[r]^h\ar[d]^{-l}& B\ar[r]^{mf}\ar[d]^f& E\ar[r]\ar@{=}[d]& 0\\
0\ar[r]& C\ar[r]^g& D\ar[r]^m& E\ar[r]& 0
}
$$
which show the assertions of the lemma.
\end{proof}

We introduce our two main ambient categories.

\begin{dfn}
Let $R$ be a local ring with maximal ideal $\m$.
Denote by $\dso(R)$ the subcategory of $\ds(R)$ consisting of objects that are locally zero on the punctured spectrum of $R$, and by $\cm_0(R)$ the subcategory of $\cm(R)$ consisting of $R$-modules that are locally free on the punctured spectrum of $R$.
Thus:
$$
\begin{array}{l}
\dso(R)=\{X\in\ds(R)\mid\text{$X_\p\cong0$ for all $\m\ne\p\in\spec R$}\},\\
\cm_0(R)=\{M\in\cm(R)\mid\text{$M_\p$ is $R_\p$-free for all $\m\ne\p\in\spec R$}\}.
\end{array}
$$
\end{dfn}

\begin{rem}
Let $R$ be a local ring.
Then $R$ has an isolated singularity if and only if $\dso(R)=\ds(R)$.
When $R$ is a Cohen--Macaulay local ring, $R$ has an isolated singularity if and only if $\cm_0(R)=\cm(R)$.
\end{rem}

Now we give the precise definition of an extension-closed subcategory.
Note that in our sense an extension-closed subcategory is necessarily closed under direct summands.

\begin{dfn}
Let $\C$ be an additive category.
\begin{enumerate}[(1)]
\item
We say that a subcategory $\X$ of $\C$ is {\em closed under finite direct sums} provided that if $X_1,\dots,X_n$ are a finite number of objects of $\C$ that belong to $\X$, then the direct sum $X_1\oplus\cdots\oplus X_n$ belongs to $\X$ as well.
\item
We say that a subcategory $\X$ of $\C$ is {\em closed under direct summands} provided that if $X$ is an object of $\C$ that belongs to $\X$ and $Y$ is an object of $\C$ which is a direct summand of $X$, then $Y$ belongs to $\X$.
\item
An {\em additively closed} subcategory of $\C$ is defined to be a subcategory of $\C$ which is closed under finite direct sums and direct summands.
\item
Assume that $\C$ is extriangulated in the sense of \cite{NP}.
Let $\X$ be a subcategory of $\C$ closed under direct summands.
We say that $\X$ is {\em extension-closed} provided that for every conflation $L\to M\to N$ in $\C$, if $L$ and $N$ belong to $\X$, then so does $M$.
\end{enumerate}
\end{dfn}

\begin{rem}
\begin{enumerate}[(1)]
\item
Since $\dso(R)$ is a triangulated category, it is an extriangulated category.
A subcategory $\X$ of $\dso(R)$ closed under direct summands is extension-closed if and only if for each exact triangle $L\to M\to N\rightsquigarrow$ in $\dso(R)$ with $L,N\in\X$ one has $M\in\X$.
\item
Let $R$ be a Cohen--Macaulay local ring.
Then $\cm_0(R)$ is an exact category whose conflations are the short exact sequences of maximal Cohen--Macaulay $R$-modules that are locally free on the punctured spectrum.
Hence $\cm_0(R)$ is an extriangulated category.
A subcategory $\X$ of $\cm_0(R)$ closed under direct summands is extension-closed if and only if for an exact sequence $0\to L\to M\to N\to0$ of modules in $\cm_0(R)$ with $L,N\in\X$ one has $M\in\X$.
\end{enumerate}
\end{rem}

Next we introduce some notations.

\begin{dfn}
Let $\C$ be an additive category, and let $\X$ be a subcategory of $\C$.
\begin{enumerate}[(1)]
\item
We denote by $\ind_\C\X$ the set of isomorphism classes of indecomposable objects in $\X$.
\item
Denote by $\add_\C\X$ the {\em additive closure} of $\X$, which is defined as the smallest additively closed subcategory of $\C$ containing $\X$.
Hence $\add_\C\X$ consists of the direct summands of finite direct sums of objects in $\X$.
\item
Assume that $\C$ is extriangulated.
We denote by $\ext_\C\X$ the {\em extension closure} of $\X$, which is by definition the smallest extension-closed subcategory of $\C$ containing $\X$.
\end{enumerate}
When $\X$ is the subcategory of $\C$ defined by a single object $X$, we may write $\add_\C X$ and $\ext_\C X$ instead of $\add_\C\X$ and $\ext_\C\X$, respectively.
Note that $\add_\C X$ consists of the direct summands of finite direct sums of copies of $X$.
Let $\E(R),\E_0(R)$ be the sets of extension-closed subcategories of $\cm(R),\cm_0(R)$ respectively.
Note that $\E(R),\E_0(R)$ are posets (partially ordered sets) with respect to the inclusion relation $\subseteq$.
\end{dfn}

We recall the definitions of finite and countable CM-representation types.

\begin{dfn}
Let $R$ be a Cohen--Macaulay local ring.
We say that $R$ has {\em finite CM-representation type} if $\ind\cm(R)$ is finite.
We say that $R$ has {\em countable CM-representation type} if $\ind\cm(R)$ is countably infinite.
\end{dfn}

The following lemma states some properties of additive and extension closures which will often be used.

\begin{lem}\label{3}
\begin{enumerate}[\rm(1)]
\item
Let $0\to L\to M\to N\to0$ be an exact sequence of $R$-modules.
Let $\X$ be a subcategory of $\mod R$.
Suppose that $L,N$ belong to $\add\X$.
Then there exists an exact sequence $0\to A \to B\to C\to0$ of $R$-modules such that $A,C$ are finite direct sums of modules in $\X$ and $M$ is a direct summand of $B$.
\item
Let $M$ and $N$ be $R$-modules, and let $S$ be a multiplicatively closed subset of $R$.
Suppose that $M$ belongs to the extension closure $\ext_{\mod R}N$.
Then the localization $M_S$ belongs to the extension closure $\ext_{\mod R_S}N_S$.
\item
Let $R$ be a Gorenstein local ring.
Let $\X$ and $\Y$ be subcategories of $\cm(R)$.
If the equality $\ext\X=\add\Y$ holds, then the equality $\ext(\X\cup\{R\})=\add(\Y\cup\{R\})$ holds as well.
\end{enumerate}
\end{lem}

\begin{proof}
(1) There exist $R$-modules $L',N'$ and finite direct sums $A,C$ of $R$-modules in $\X$ such that $L\oplus L'\cong A$ and $N\oplus N'\cong C$.
Taking the direct sum with the trivial exact sequences $0\to L'\to L'\to0\to0$ and $0\to0\to N'\to N'\to0$, we get an exact sequence $0\to L\oplus L'\to L'\oplus M\oplus N'\to N\oplus N'\to0$.
Setting $B=L'\oplus M\oplus N'$, we obtain such an exact sequence $0\to A\to B\to C\to0$ as in the assertion.

(2) Let $\X$ be the subcategory of $\mod R$ consisting of $R$-modules $X$ such that $X_S\in\ext_{\mod R_S}N_S$.
Then $\X$ contains $N$.
Let $Z$ be an $R$-module in $\X$ and $W$ is a direct summand of $Z$.
Then $Z_S$ is in $\ext_{\mod R_S}N_S$ and $W_S$ is a direct summand of $Z_S$.
Hence $W_S$ is in $\ext_{\mod R_S}N_S$, so that $W$ is in $\X$.
Let $0\to X\to E\to Y\to0$ be an exact sequence in $\mod R$ such that $X,Y\in\X$.
Then there is an exact sequence $0\to X_S\to E_S\to Y_S\to0$ in $\mod R_S$ and $X_S,Y_S$ belong to $\ext_{\mod R_S}N_S$.
Hence $E_S$ belongs to $\ext_{\mod R_S}N_S$, and we see that $E$ is in $\X$.
The subcategory $\X$ of $\mod R$ contains $N$ and is extension-closed.
It follows that $\ext N$ is contained in $\X$.
Since $M$ is in $\ext N$, we observe that $M_S$ is in $\ext_{\mod R_S}N_S$.

(3) We have $R\in\add(\Y\cup\{R\})\supseteq\add\Y=\ext\X\supseteq\X$.
Hence $\add(\Y\cup\{R\})$ contains $\X\cup\{R\}$, and is closed under direct summands by definition.
Let $0\to L\to M\to N\to0$ be an exact sequence in $\mod R$ with $L,N\in\add(\Y\cup\{R\})$.
By assertion (1), we get an exact sequence $0\to A\oplus R^{\oplus a}\to B\to C\oplus R^{\oplus c}\to0$ with $a,c\in\N$ and $A,C\in\add\Y$ such that $M$ is a direct summand of $B$.
We have pushout and pullback diagrams:
$$
\xymatrix@R-1pc@C-1pc{
&0\ar[d]&0\ar[d]\\
&R^{\oplus a}\ar@{=}[r]\ar[d]&R^{\oplus a}\ar[d]\\
0\ar[r]&A\oplus R^{\oplus a}\ar[r]\ar[d]& B\ar[r]\ar[d]& C\oplus R^{\oplus c}\ar[r]\ar@{=}[d]&0\\
0\ar[r]&A\ar[r]\ar[d]&V\ar[r]\ar[d]&C\oplus R^{\oplus c}\ar[r]&0\\
&0&0
}\qquad
\xymatrix@R-1pc@C-1pc{
&&0\ar[d]&0\ar[d]\\
0\ar[r]&A\ar[r]\ar@{=}[d]&W\ar[r]\ar[d]&C\ar[r]\ar[d]&0\\
0\ar[r]&A\ar[r]&V\ar[d]\ar[r]&C\oplus R^{\oplus c}\ar[r]\ar[d]&0\\
&&R^{\oplus c}\ar[d]\ar@{=}[r]&R^{\oplus c}\ar[d]\\
&&0&0}
$$
As $R$ is Gorenstein and $V$ is maximal Cohen--Macaulay, the middle column in the left diagram splits and we get $B\cong R^{\oplus a}\oplus V$.
The middle column in the right diagram splits, so $V\cong W\oplus R^{\oplus c}$.
Hence $B\cong W\oplus R^{\oplus(a+c)}$.
As $\add\Y=\ext\X$ is extension-closed and contains $A,C$, the first row in the right diagram shows $W\in\add\Y$.
Therefore, $B$ is in $\add(\Y\cup\{R\})$, and so is $M$.
We have shown that $\add(\Y\cup\{R\})$ contains $\ext(\X\cup\{R\})$.
Let $\ZZ$ be an extension-closed subcategory of $\mod R$ containing $\X\cup\{R\}$.
Since $\ZZ$ contains $\X$ and is extension-closed, $\ZZ$ contains $\ext\X=\add\Y$.
Hence $\ZZ$ contains $\add(\Y\cup\{R\})$.
It follows that $\add(\Y\cup\{R\})=\ext(\X\cup\{R\})$.
\end{proof}

Next we recall some notions from commutative algebra.

\begin{dfn}
Let $R$ be a local ring.
\begin{enumerate}[(1)]
\item
We say that $R$ is a {\em hypersurface} if it satisfies the inequality $\edim R-\depth R\le1$, where $\edim R$ and $\depth R$ stand for the embedding dimension of $R$ and the depth of $R$, respectively.
\item
Let $M$ be an $R$-module.
For an integer $n>0$, we denote by $\syz^nM$ the {\em $n$th syzygy} of $M$, that is, the image of the $n$th differential map in a minimal free resolution of $M$.
We put $\syz M=\syz^1M$ and $\syz^0M=M$.
\item
An $R$-module $M$ is called {\em rigid} if $\Ext_R^1(M,M)=0$.
\end{enumerate}
\end{dfn}

\begin{rem}
\begin{enumerate}[(1)]
\item
A local ring $(R,\m)$ is a hypersurface if and only if the $\m$-adic completion of $R$ is isomorphic to $S/(f)$ for some regular local ring $(S,\n)$ and some element $f\in\n$; see \cite[5.1]{A}.
\item
Let $R$ be a hypersurface.
Let $M$ be a nonfree maximal Cohen--Macaulay $R$-module with no free summand.
Then there is an isomorphism $\syz^2M\cong M$; see \cite[5.1.1 and 5.1.2]{A}.
\item
For an $R$-module $M$ and an integer $n\ge0$, the $n$th syzygy $\syz^nM$ is uniquely determined up to isomorphism, since so is a minimal free resolution of $M$. 
\end{enumerate}
\end{rem}

We investigate the relationship between extension closures and syzygies over a local ring.

\begin{lem}\label{4}
Let $(R,\m,k)$ be a local ring of (Krull) dimension $d$.
Let $M$ be an $R$-module.
\begin{enumerate}[\rm(1)]
\item
Let $N$ be an $R$-module.
If $N$ belongs to $\ext M$, then $\syz N$ belongs to $\ext\syz M$. 
\item
Suppose that the local ring $R$ is Gorenstein and singular, and that the $R$-module $M$ is maximal Cohen--Macaulay.
If $\ext M=\cm_0(R)$, then $\ext\syz^iM=\cm_0(R)$ for all integers $i\ge0$.
\end{enumerate}
\end{lem}

\begin{proof}
(1) Let $\X$ be the subcategory of $\mod R$ consisting of $R$-modules $X$ such that $\syz X$ is in $\ext\syz M$.
Then $M$ is in $\X$.
Let $X$ be an $R$-module in $\X$ and $Y$ a direct summand of $X$.
Then $\syz Y$ is a direct summand of $\syz X$.
Since $\syz X$ is in $\ext\syz M$, so is $\syz Y$.
Therefore, $Y$ belongs to $\X$.
Let $0\to A\to B\to C\to0$ be an exact sequence of $R$-modules with $A,C\in\X$.
Then there is an exact sequence $0\to\syz A\to\syz B\oplus R^{\oplus n}\to\syz C\to0$ with $n\in\N$.
Since $\syz A,\syz C$ are in $\ext\syz M$, so is $\syz B$.
Hence $B$ is in $\X$.
Thus, $\X$ is extension-closed and contains $M$.
It follows that $\X$ contains $\ext M$, and $N$ belongs to $\X$, which means that $\syz N$ is in $\ext\syz M$.

(2) We may assume $i=1$.
The module $M$ is in $\ext M=\cm_0(R)$, so that $\syz M$ is also in $\cm_0(R)$, whence $\ext\syz M\subseteq\cm_0(R)$.
As $R$ is Gorenstein and $M$ is maximal Cohen--Macaulay, there exists a maximal Cohen--Macaulay $R$-module $N$ such that $\syz^dk\cong\syz N\oplus R^{\oplus n}$ for some $n\in\N$.
As $R$ is singular, we have $n=0$ by \cite[Corollary 1.3]{D}.
Since $\syz^dk$ belongs to $\cm_0(R)$, so does $\syz N$.
Since $N$ is a maximal Cohen--Macaulay module over a Gorenstein ring $R$, it is observed that $N$ is in $\cm_0(R)$.
As $\cm_0(R)=\ext M$, we see from (1) that $\syz N$ belongs to $\ext\syz M$.
Hence $\syz^dk$ is in $\ext\syz M$.
It follows from \cite[Corollary 2.6]{stcm} that $\ext\syz M=\cm_0(R)$.
\end{proof}

We study extension-closedness of additive closures of rigid maximal Cohen--Macaulay modules.

\begin{lem}\label{2}
Let $R$ be a Cohen--Macaulay local ring.
Let $M$ be a rigid maximal Cohen--Macaulay $R$-module.
\begin{enumerate}[\rm(1)]
\item
The additive closure $\add M$ is an extension-closed subcategory of $\cm(R)$.
\item
If $R$ is Gorenstein, then $\add\{R,M\}$ is an extension-closed subcategory of $\cm(R)$.
\item
If $M$ has no nonzero free summand, then $\add M\ne\add\{R,M\}$.
\end{enumerate}
\end{lem}

\begin{proof}
(1) By definition, $\add M$ is closed under direct summands.
Let $0\to X\to Y\to Z\to0$ be an exact sequence of maximal Cohen--Macaulay $R$-modules, and suppose that $X$ and $Z$ are in $\add M$.
By Lemma \ref{3}(1) there is an exact sequence $\sigma:0\to M^{\oplus a}\to W\to M^{\oplus b}\to0$ in $\mod R$ with $a,b\in\N$ such that $Y$ is a direct summand of $W$.
Since $M$ is rigid, $\Ext_R^1(M^{\oplus b},M^{\oplus a})=0$.
Hence the short exact sequence $\sigma$ splits, and $W$ is isomorphic to $M^{\oplus(a+b)}$.
Therefore, $Y$ belongs to $\add M$.
It follows that $\add M$ is extension-closed.

(2) If $R$ is Gorenstein, then the maximal Cohen--Macaulay $R$-module $R\oplus M$ is rigid.
It follows from (1) that $\add\{R,M\}=\add(R\oplus M)$ is an extension-closed subcategory of $\cm(R)$.

(3) Assume $\add M=\add\{R,M\}$.
Then $R$ is a direct summand of a finite direct sum of copies of $M$.
By \cite[Corollaries 1.10 and 1.15]{LW}, we see that $R$ is a direct summand of $M$.
Thus the assertion follows.
\end{proof}

The following lemma will be used in the case where $S$ is a regular local ring.

\begin{lem}\label{1}
Let $S$ be a local ring.
Let $x,y$ be a regular sequence on $S$.
Take the quotient ring $R=S/(xy)$.
Then the $R$-modules $R/(x)$ and $R/(y)$ are rigid. 
\end{lem}

\begin{proof}
Using the fact that the elements $x$ and $y$ are regular on $S$, we observe that the equalities $0:_Rx=yR$ and $0:_Ry=xR$ hold true.
This implies that the sequence $\cdots\xrightarrow{y}R\xrightarrow{x}R\xrightarrow{y}R\xrightarrow{x}R\to R/xR\to0$ is exact.
We have $\Ext_R^1(R/xR,R/xR)=\h^1(0\to R/xR\xrightarrow{0}R/xR\xrightarrow{y}R/xR\xrightarrow{0}R/xR\xrightarrow{y}\cdots)=0$, since $y$ is regular on $S/xS=R/xR$.
Hence $R/xR$ is a rigid $R$-module.
By symmetry, the $R$-module $R/yR$ is also rigid.
\end{proof}

Here we recall the definitions of the stable category and cosyzygies of maximal Cohen--Macaulay modules.

\begin{dfn}
\begin{enumerate}[(1)]
\item
Let $R$ be a Cohen--Macaulay local ring.
Let $\lcm(R)$ be the {\em stable category} of $\cm(R)$; the objects of $\lcm(R)$ are the same as those of $\cm(R)$, and the hom-set from $M$ to $N$ is given by
$$
\Hom_{\lcm(R)}(M,N)=\Hom_R(M,N)/\{f\in\Hom_R(M,N)\mid\text{$f$ factors through a free $R$-module}\}.
$$
\item
Let $R$ be a Gorenstein local ring.
Let $M$ be a maximal Cohen--Macaulay $R$-module.
We define the {\em (first) cosyzygy} $\syz^{-1}M$ of $M$ as $(\syz(M^\ast))^\ast$, where $(-)^\ast=\Hom_R(-,R)$.
Note then that $\syz^{-1}M$ is maximal Cohen--Macaulay.
For $n\ge2$, we define the {\em $n$th cosyzygy} $\syz^{-n}M$ of $M$ inductively by $\syz^{-n}M=\syz^{-1}(\syz^{-(n-1)}M)$.
Note that $\syz^{-n}M=0$ for every integer $n>0$ when $M$ is a free $R$-module.
\end{enumerate}
\end{dfn}

\begin{rem}
\begin{enumerate}[(1)]
\item
Let $R$ be a Cohen--Macaulay local ring.
For each integer $n\ge0$, taking the $n$th syzygy defines an additive functor $\syz^n:\lcm(R)\to\lcm(R)$.
\item
Let $R$ be a Gorenstein local ring.
For each $n>0$, taking the $n$th cosyzygy defines an additive functor $\syz^{-n}:\lcm(R)\to\lcm(R)$.
The functors $\syz^n,\syz^{-n}:\lcm(R)\to\lcm(R)$ are mutually quasi-inverse equivalences.
The stable category $\lcm(R)$ is a triangulated category, and the assignment $M\mapsto M$ gives a triangle equivalence $\lcm(R)\to\ds(R)$.
For the details, we refer the reader to \cite[Theorem 4.4.1]{B}.
\end{enumerate}
\end{rem}

The lemma below follows from \cite[Proposition (3.11)]{Y}.
For the definition and fundamental properties of the Auslander--Reiten translation functor, we refer the reader to \cite[Chapters 3 and 5]{Y}.

\begin{lem}\label{6}
Let $R$ be a Gorenstein local ring of dimension $d$.
Let $\tau:\lcm(R)\to\lcm(R)$ stand for the Auslander--Reiten translation functor.
Then $\tau\cong\syz^{2-d}$.
\end{lem}

Next we introduce the syzygy of a subcategory of modules.

\begin{dfn}
Let $R$ be a local ring.
For a subcategory $\X$ of $\mod R$ we denote by $\syz\X$ the subcategory of $\mod R$ consisting of $R$-modules $Y$ such that there exists an exact sequence $0\to Y\to P\to X\to0$ in $\mod R$ with $P\in\add R$ and $X\in\X$.
Note that $\syz\X$ necessarily contains $\add R$.
\end{dfn}

The lemma below describes commutativity of the syzygy and extension closure of a subcategory of modules.

\begin{lem}\label{10}
Let $R$ be a Gorenstein local ring.
Let $\X$ be a subcategory of $\mod R$ contained in $\cm(R)$.
Then there is an equality $\ext(\syz\X)=\add\syz(\ext\X)$.
\end{lem}

\begin{proof}
We have that $\add\syz(\ext\X)$ contains $\syz\X$ and is closed under direct summands.
Let $0\to L\to M\to N\to0$ be an exact sequence in $\mod R$ with $L,N\in\add\syz(\ext\X)$.
Lemma \ref{3}(1) implies that there exists an exact sequence $0\to A\to B\to C\to0$ in $\mod R$ with $A,C\in\syz(\ext\X)$ such that $M$ is a direct summand of $B$.
There exist exact sequences $0\to A\to F\to V\to0$ and $0\to C\to G\to W\to0$ with $F,G\in\add R$ and $V,W\in\ext\X$.
We have a pushout diagram as in the left below.
The middle row splits since $R$ is Gorenstein and $C$ is maximal Cohen--Macaulay.
Hence we get another pushout diagram as in the right below. 
$$
\xymatrix@R-1pc@C-1pc{
&0\ar[d]&0\ar[d]\\
0\ar[r]& A\ar[r]\ar[d]& B\ar[r]\ar[d]& C\ar[r]\ar@{=}[d]& 0\\
0\ar[r]&F\ar[r]\ar[d]& D\ar[r]\ar[d]& C\ar[r]&0\\
&V\ar@{=}[r]\ar[d]&V\ar[d]\\
&0&0
}\qquad\qquad
\xymatrix@R-1pc@C-1pc{
&&0\ar[d]&0\ar[d]\\
0\ar[r]&B\ar[r]\ar@{=}[d]&F\oplus C\ar[r]\ar[d]&V\ar[r]\ar[d]&0\\
0\ar[r]&B\ar[r]&F\oplus G\ar[d]\ar[r]&E\ar[r]\ar[d]&0\\
&&W\ar[d]\ar@{=}[r]&W\ar[d]\\
&&0&0}
$$
The right column in the right diagram shows that $E$ is in $\ext\X$, and then the middle row in the same diagram shows that $B$ is in $\syz(\ext\X)$.
Hence $M$ is in $\add\syz(\ext\X)$, and therefore, $\add\syz(\ext\X)$ is extension-closed.

It remains to prove that if $\Y$ is an extension-closed subcategory of $\mod R$ containing $\syz\X$, then $\Y$ contains $\add\syz(\ext\X)$.
Let $\ZZ$ be the subcategory of $\mod R$ consisting of $R$-modules $Z$ with $\syz Z\in\Y$.
Then $\ZZ$ contains $\X$.
If an $R$-module $Z$ is in $\ZZ$ and $W$ is a direct summand of $Z$, then $\syz W$ is a direct summand of $\syz Z$, which belongs to $\Y$, and so does $\syz W$ since $\Y$ is closed under direct summands.
Hence $\ZZ$ is closed under direct summands.
Let $0\to A\to B\to C\to0$ be an exact sequence with $A,C\in\ZZ$.
Then there is an exact sequence $0\to\syz A\to\syz B\oplus F\to\syz C\to0$ with $F\in\add R$ and $\syz A,\syz C\in\Y$.
As $\Y$ is extension-closed, $\syz B$ belongs to $\Y$, and hence $B$ is in $\ZZ$.
It follows that $\ZZ$ is extension-closed.
We now observe that $\ZZ$ contains $\ext\X$, which implies $\syz(\ext\X)\subseteq\Y$.
As $\Y$ is extension-closed, we get the desired inclusion $\add\syz(\ext\X)\subseteq\Y$.
\end{proof}

\section{On the $\a_n,\d_n,\e_6,\e_7,\e_8$-singularities with dimension at most two}

In this section, we give a complete classification of the extension-closed subcategories of $\cm(R)$ in the case where $R$ is a $P$-singularity with $P\in\{\a_n,\d_n,\e_6,\e_7,\e_8\}$ and has dimension at most two.

\begin{thm}\label{22}
Let $R=k[\![x,y]\!]/(x^2+y^{n+1})$ be the $\a_n^1$-singularity, where $n$ is a positive integer.
When $n$ is odd, put $N_\pm=R/(y^{(n+1)/2}\pm\sqrt{-1}\,x)$.
The Hasse diagram of the poset $\E(R)$ is
$$
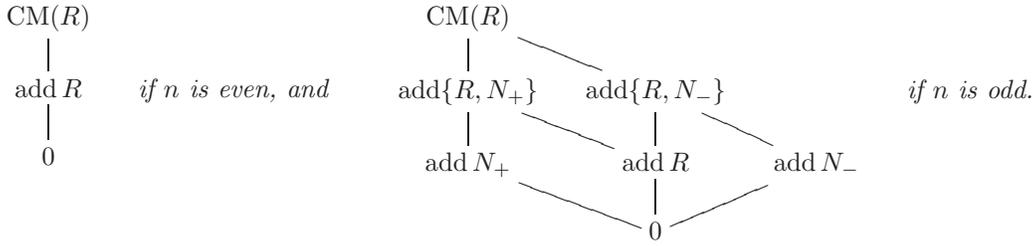
\xymatrix@R-1pc@C-1pc{
\cm(R)\ar@{-}[d]\\
\add R\ar@{-}[d]&\text{if $n$ is even, and}\\
0
}
\qquad
\xymatrix@R-1pc@C-1pc{
\cm(R)\ar@{-}[d]\ar@{-}[rd]\\
\add\{R,N_+\}\ar@{-}[d]\ar@{-}[rd]&\add\{R,N_-\}\ar@{-}[rd]\ar@{-}[d]&&\text{if $n$ is odd.}
\\
\add N_+\ar@{-}[rd]&\add R\ar@{-}[d]&\add N_-\ar@{-}[ld]\\
&0
}
$$
\end{thm}

\begin{proof}
We begin with the case where $n$ is even.
Put $l=n/2$, $I_0=R$ and $I_j=(x,y^j)R$ for each $1\le j\le l$.
By \cite[(5.12)]{Y} we have $\ind\cm(R)=\{I_0,I_1,\dots,I_l\}$ and there exist exact sequences $0\to I_j\to I_{j-1}\oplus I_{j+1}\to I_j\to0$ for each $1\le j\le l-1$, and $0\to I_l\to I_{l-1}\oplus I_l\to I_l\to0$.
Hence $R=I_0\in\ext I_1=\ext I_2=\cdots=\ext I_l$.
It is observed that $\ext I_j=\cm(R)$ for all $1\le j\le l$.
We obtain the equality $\E(R)=\{0,\add R,\cm(R)\}$.

Next we deal with the case where $n$ is odd.
Put $l=(n-1)/2$, $N_\pm=R/(y^{l+1}\pm\sqrt{-1}\,x)\cong(y^{l+1}\mp\sqrt{-1}\,x)R$, $M_0=R$ and $M_j=\cok\left(\begin{smallmatrix}x&y^j\\y^{n+1-j}&-x\end{smallmatrix}\right)\cong(x,y^j)R$ for each integer $1\le j\le l$.
Using \cite[(9.9)]{Y}, we have that $\ind\cm(R)=\{M_0,M_1,\dots,M_l,N_+,N_-\}$ and there exist exact sequences
$$
\begin{array}{l}
0\to M_j\to M_{j-1}\oplus M_{j+1}\to M_j\to0\text{ for each $1\le j\le l-1$,}\\
0\to M_l\to M_{l-1}\oplus N_+\oplus N_-\to M_l\to0\text{ and }0\to N_\pm\to M_l\to N_\mp\to0.
\end{array}
$$
Hence $R=M_0\in\ext M_1=\ext M_2=\cdots=\ext M_l\ni N_\pm$ and $M_l\in\ext\{N_+,N_-\}$.
We get $\ext M_1=\ext M_2=\cdots=\ext M_l=\ext\{N_+,N_-\}=\cm(R)$.
Applying Lemma \ref{1}, we see that $N_+,N_-$ are rigid $R$-modules.
By Lemma \ref{2}, we conclude that $\E(R)=\{0,\add R,\add N_+,\add N_-,\add\{R,N_+\},\add\{R,N_-\},\cm(R)\}$.
\end{proof}

\begin{thm}\label{23}
Let $R=k[\![x,y]\!]/(x^3+y^4)$ be the $\e_6^1$-singularity.
The Hasse diagram of $\E(R)$ is the following.
$$
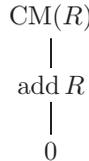
\xymatrix@R-1pc@C-1pc{
\cm(R)\ar@{-}[d]\\
\add R\ar@{-}[d]\\
0
}
$$
\end{thm}

\begin{proof}
By virtue of \cite[(9.13)]{Y}, we have that $\ind\cm(R)=\{R,M_1,M_2,N_1,A,B,X\}$ with $N_1\cong\m:=(x,y)R$, $B\cong\m^2$, and there exist exact sequences
$$
\begin{array}{l}
a:0\to M_2\to X\to M_2\to0,\quad
b:0\to X\to M_2\oplus A\oplus B\to X\to0,\quad
c:0\to A\to X\oplus N_1\to B\to0,\\
d:0\to B\to X\oplus M_1\to A\to0,\quad
e:0\to M_1\to A\to N_1\to0,\quad
f:0\to N_1\to B\oplus R\to M_1\to0
\end{array}
$$
such that the maps $C\to D$ with $C,D\in\ind\cm(R)$ appearing in different exact sequences (e.g., the maps $X\to M_2$ in $a,b$) are the same.
It follows from \cite[Corollary 2.7]{stcm} that $\ext N_1=\cm(R)$.
The associated graded ring $\gr_\m R$ is isomorphic to $k[x,y]/(x^3)$, which has depth one.
We see from \cite[Theorem 5.5]{serre} that $\ext B=\cm(R)$.
The exact sequences $a$ and $b$ show that $X\in\ext M_2$ and $B\in\ext X$, respectively.
Hence the equalities $\ext M_2=\ext X=\ext B=\cm(R)$ hold.
The exact sequences $b,c,d,f$ make exact squares as in the left below, while $b,c,e,f$ make exact squares as in the right below.
$$
\xymatrix@R-1pc@C2pc{
A\ar[r]\ar[d]& X\ar[r]\ar[d]& M_2\oplus A\ar[d]\\
N_1\ar[r]\ar[d]& B\ar[r]\ar[d]& X\ar[d]\\
R\ar[r]& M_1\ar[r]& A
}\qquad
\xymatrix@R-1pc@C2pc{
M_1\ar[r]\ar[d]& A\ar[r]\ar[d]& X\ar[r]\ar[d]& M_2\oplus A\ar[r]\ar[d]& M_2\oplus N_1\ar[r]\ar[d]& M_2\oplus R\ar[d]\\
0\ar[r]& N_1\ar[r]& B\ar[r]& X\ar[r]& B\ar[r]& M_1
}$$
In view of Lemma \ref{5}, there exist short exact sequences $0\to A\to M_2\oplus A\oplus R\to A\to0$ and $0\to M_1\to M_2\oplus R\to M_1\to0$, which show that $M_2\in\ext A\cap\ext M_1$.
It is seen that the equalities $\ext A=\ext M_1=\ext M_2=\cm(R)$ hold.
Now it is observed that $\E(R)=\{0,\add R,\cm(R)\}$.
\end{proof}

\begin{thm}\label{24}
Let $R=k[\![x,y]\!]/(x^3+y^5)$ be the $\e_8^1$-singularity.
The Hasse diagram of $\E(R)$ is the following.
$$
\xymatrix@R-1pc@C-1pc{
\cm(R)\ar@{-}[d]\\
\add R\ar@{-}[d]\\
0
}
$$
\end{thm}

\begin{proof}
Put $U=\left(\begin{smallmatrix}y&-x&0\\0&y&-x\\x&0&y^3\end{smallmatrix}\right)$ and $V=\left(\begin{smallmatrix}y&-x&0\\0&y^2&-x\\x&0&y^2\end{smallmatrix}\right)$.
In view of \cite[(9.15)]{Y} and Lemma \ref{6}, we have that
$$
\ind\cm(R)=\{R,M_1,M_2,N_1,N_2,A_1,A_2,B_1,B_2,C_1,C_2,D_1,D_2,X_1,X_2,Y_1,Y_2\},
$$
where $A_1=\cok U$, $A_2=\cok V$, $M_1\cong(x^2,y)R$, $M_2\cong(x^2,y^2)R$, $N_1\cong\m:=(x,y)R$, $N_2\cong(x,y^2)R$, $N_i\cong\syz M_i$, $B_i\cong\syz A_i$, $D_i\cong\syz C_i$, $Y_i\cong\syz X_i$ for $i=1,2$, and there exist exact sequences
$$
\begin{array}{l}
0\to Y_2\to D_1\oplus X_1\to X_2\to0,\quad
0\to X_1\to X_2\oplus A_2\oplus C_2\to Y_1\to0,\quad
0\to D_1\to A_1\oplus X_2\to C_1\to0,\\
0\to X_2\to C_1\oplus Y_1\to Y_2\to0,\quad
0\to C_2\to Y_1\oplus N_2\to D_2\to0,\quad
0\to Y_1\to Y_2\oplus B_2\oplus D_2\to X_1\to0,\\
0\to C_1\to B_1\oplus Y_2\to D_1\to0,\quad
0\to B_1\to N_1\oplus D_1\to A_1\to0,\quad
0\to N_2\to D_2\to M_2\to0,\\
0\to D_2\to X_1\oplus M_2\to C_2\to0,
\end{array}
$$
where the maps $E\to F$ with $E,F\in\ind\cm(R)$ are all the same.
It is verified that $R\xleftarrow{\left(\begin{smallmatrix}xy&x^2&-y^2\end{smallmatrix}\right)}R^{\oplus3}\xleftarrow{U}R^{\oplus3}$ and $R\xleftarrow{\left(\begin{smallmatrix}xy^2&x^2&-y^3\end{smallmatrix}\right)}R^{\oplus3}\xleftarrow{V}R^{\oplus3}$ are exact sequences, so $A_1\cong\m^2$ and $A_2\cong(x^2,xy^2,y^3)R$.
For each $I\in\{M_1,M_2,N_1,N_2,A_1,A_2\}$, by \cite[Proposition 6.8]{serre} and completion $R/\m$ is in $\ext R/I$.
Lemma \ref{4}(1) (or \cite[Lemma 5.2(1)]{serre}) implies $\m\in\ext I$.
This gives $\ext I=\cm(R)$ by \cite[Corollary 2.7]{stcm}.
It follows from Lemma \ref{4}(2) that $\ext B_1=\ext B_2=\cm(R)$.
The above exact sequences make the following exact squares.
$$
\xymatrix@R-1pc@C-1pc{
Y_2\ar[r]\ar[d]& D_1\ar[r]\ar[d]& A_1\ar[d]\\
X_1\ar[r]\ar[d]& X_2\ar[r]\ar[d]& C_1\ar[d]\\
A_2\oplus C_2\ar[r]& Y_1\ar[r]& Y_2
}\qquad
\xymatrix@R-1pc@C-1pc{
X_1\ar[r]\ar[d]& X_2\ar[r]\ar[d]& C_1\oplus B_2\ar[d]\\
A_2\oplus C_2\ar[r]\ar[d]& Y_1\ar[r]\ar[d]& Y_2\oplus B_2\ar[d]\\
A_2\oplus N_2\ar[r]& D_2\ar[r]& X_1
}\qquad
\xymatrix@R-1pc@C-1pc{
C_1\ar[r]\ar[d]& B_1\ar[r]\ar[d]& N_1\ar[d]\\
Y_2\ar[r]\ar[d]& D_1\ar[r]\ar[d]& A_1\ar[d]\\
X_1\ar[r]& X_2\ar[r]& C_1
}\quad
\xymatrix@R-1pc@C-1pc{
C_2\ar[r]\ar[d]& Y_1\ar[r]\ar[d]& Y_2\oplus B_2\ar[d]\\
N_2\ar[r]\ar[d]& D_2\ar[r]\ar[d]& X_1\ar[d]\\
0\ar[r]& M_2\ar[r]& C_2
}$$
By Lemma \ref{5}, these exact squares produce the following four exact sequences:
$$
\begin{array}{l}
0\to Y_2\to A_1\oplus A_2\oplus C_2\to Y_2\to0,\quad
0\to X_1\to C_1\oplus B_2\oplus A_2\oplus N_2\to X_1\to0,\\
0\to C_1\to N_1\oplus X_1\to C_1\to0,\quad
0\to C_2\to Y_2\oplus B_2\to C_2\to0.
\end{array}
$$
Hence $A_2\in\ext Y_2\cap\ext X_1$, $X_1\in\ext C_1$ and $B_2\in\ext C_2$, and we get $\ext I=\cm(R)$ for $I\in\{Y_2,X_1,C_1,C_2\}$.
Lemma \ref{4}(2) implies $\ext I=\cm(R)$ for $I\in\{X_2,Y_1,D_1,D_2\}$.
Consequently, $\ext I=\cm(R)$ for all $I\in\{M_1,M_2,N_1,N_2,A_1,A_2,B_1,B_2,Y_2,X_1,C_1,C_2,X_2,Y_1,D_1,D_2\}$.
We obtain $\E(R)=\{0,\add R,\cm(R)\}$.
\end{proof}

\begin{thm}\label{25}
Let $R=k[\![x,y]\!]/(x^3+xy^3)$ be the $\e_7^1$-singularity.
The Hasse diagram of $\E(R)$ is the following.
$$
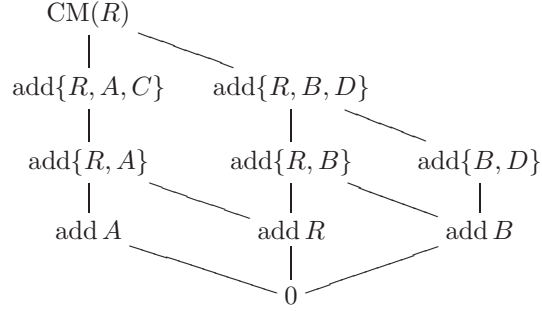
\xymatrix@R-1pc@C-1pc{
\cm(R)\ar@{-}[d]\ar@{-}[rd]\\
\add\{R,A,C\}\ar@{-}[d]&\add\{R,B,D\}\ar@{-}[d]\ar@{-}[rd]\\
\add\{R,A\}\ar@{-}[d]\ar@{-}[rd]&\add\{R,B\}\ar@{-}[d]\ar@{-}[rd]&\add\{B,D\}\ar@{-}[d]\\
\add A\ar@{-}[rd]&\add R\ar@{-}[d]&\add B\ar@{-}[ld]\\
&0
}
$$
Here, $A=R/(x)$, $B=\syz A=R/(x^2+y^3)$, $C=\cok\left(\begin{smallmatrix}x^2&xy\\xy^2&-x^2\end{smallmatrix}\right)$ and $D=\syz C=\cok\left(\begin{smallmatrix}x&y\\y^2&-x\end{smallmatrix}\right)$.
\end{thm}

\begin{proof}
We observe from \cite[(9.14)]{Y} and Lemma \ref{6} that
$$
\ind\cm(R)=\{R,A,B,C,D,M_i,N_i,X_j,Y_j\mid i=1,2\text{ and }j=1,2,3\},
$$
where $A=R/(x)$, $\syz A\cong B=R/(x^2+y^3)$, $C=\cok\left(\begin{smallmatrix}x^2&xy\\xy^2&-x^2\end{smallmatrix}\right)$, $\syz C\cong D=\cok\delta$ with $\delta=\left(\begin{smallmatrix}x&y\\y^2&-x\end{smallmatrix}\right)$, $\syz M_i\cong N_i$ for $i=1,2$, $\syz X_j\cong Y_j$ for $j=1,2,3$, $N_1=\cok\psi_1$ with $\psi_1=\left(\begin{smallmatrix}x^2&y\\xy^2&-x\end{smallmatrix}\right)$, $Y_1=\cok\eta_1$ with $\eta_1=\left(\begin{smallmatrix}y&0&x\\-x&xy&0\\0&-x&y\end{smallmatrix}\right)$, and there exist exact sequences
$$
\begin{array}{l}
0\to X_3\to X_1\oplus D\oplus Y_2\to Y_3\to0,\quad
0\to Y_2\to Y_3\oplus N_2\to X_2\to0,\quad
0\to X_1\to N_1\oplus Y_3\to Y_1\to0,\\
0\to Y_3\to Y_1\oplus C\oplus X_2\to X_3\to0,\quad
0\to N_1\to R\oplus Y_1\to M_1\to0,\quad
0\to Y_1\to M_1\oplus X_3\to X_1\to0,\\
0\to M_2\to Y_2\oplus B\to N_2\to0,\qquad
\rho:0\to B\to N_2\to A\to0,\qquad
0\to N_2\to X_2\oplus A\to M_2\to0,\\
0\to X_2\to X_3\oplus M_2\to Y_2\to0,\qquad
\sigma:0\to C\to X_3\to D\to0,\qquad
\tau:0\to D\to Y_3\to C\to0,
\end{array}
$$
where the maps $F\to G$ with $F,G\in\ind\cm(R)$ are all the same.
Lemmas \ref{2} and \ref{1} give the subdiagram
$$
\xymatrix@R-1pc@C-1pc{
\add\{R,A\}\ar@{-}[d]\ar@{-}[rd]&\add\{R,B\}\ar@{-}[d]\ar@{-}[rd]\\
\add A\ar@{-}[rd]&\add R\ar@{-}[d]&\add B\ar@{-}[ld]\\
&0
}
$$
of the Hasse diagram of $\E(R)$.
It is easily verified that the sequences $R\xleftarrow{\left(\begin{smallmatrix}x&y\end{smallmatrix}\right)}R^3\xleftarrow{\psi_1}R^3$ and $R\xleftarrow{\left(\begin{smallmatrix}xy&y^2&-x^2\end{smallmatrix}\right)}R^3\xleftarrow{\eta_2}R^3$ are exact.
Hence $N_1\cong\m$ and $Y_1\cong\m^2$, where $\m$ stands for the maximal ideal of $R$.
Note that $\gr_\m R\cong k[x,y]/(x^3)$.
The extension closures $\ext N_1$ and $\ext Y_1$ coincide with $\cm(R)$ by \cite[Corollary 2.7]{stcm} and \cite[Theorem 5.5]{serre}, respectively.
Lemma \ref{4}(2) yields $\ext M_1=\ext X_1=\cm(R)$.
There exist exact squares
$$
\xymatrix@R-1pc@C-1pc{
X_3\ar[r]\ar[d]& X_1\oplus D\ar[r]\ar[d]& N_1\oplus D\ar[r]\ar[d]& R\oplus D\ar[d]\\
Y_2\ar[r]\ar[d]& Y_3\ar[r]\ar[d]& Y_1\ar[r]\ar[d]& M_1\ar[d]\\
N_2\oplus C\ar[r]\ar[d]& X_2\oplus C\ar[r]\ar[d]& X_3\ar[r]\ar[d]& X_1\ar[d]\\
D\oplus A\oplus C\ar[r]& D\oplus M_2\oplus C\ar[r]& D\oplus Y_2\ar[r]& Y_3
}\qquad
\xymatrix@R-1pc@C-1pc{
X_2\ar[r]\ar[d]& M_2\ar[r]\ar[d]& B\ar[d]\\
X_3\ar[r]\ar[d]& Y_2\ar[r]\ar[d]& N_2\ar[d]\\
X_1\oplus D\ar[r]& Y_3\ar[r]& X_2
}\qquad
\xymatrix@R-1pc@C-1pc{
M_2\ar[r]\ar[d]& Y_2\ar[r]\ar[d]& Y_3\ar[d]\\
B\ar[r]\ar[d]& N_2\ar[r]\ar[d]& X_2\ar[d]\\
0\ar[r]& A\ar[r]& M_2
}
$$
Applying Lemma \ref{5}, we obtain the following four exact sequences.
$$
\begin{array}{l}
\alpha:0\to X_3\to N_1\oplus D\oplus N_2\oplus C\to X_3\to0,\qquad
\beta:0\to X_2\to B\oplus X_1\oplus D\to X_2\to0,\\
\gamma:0\to M_2\to Y_3\to M_2\to0,\qquad
\zeta:0\to X_3\to R\oplus D\oplus D\oplus A\oplus C\to Y_3\to0.
\end{array}
$$
The exact sequences $\alpha,\beta,\gamma$ respectively show that $N_1\in\ext X_3$, $X_1\in\ext X_2$ and $Y_3\in\ext M_2$.
The first two containments give equalities $\ext X_3=\ext X_2=\cm(R)$, which give equalities $\ext Y_3=\ext Y_2=\cm(R)$ by Lemma \ref{4}(2).
The third containment now shows that $\ext M_2=\cm(R)$, which implies $\ext N_2=\cm(R)$ by Lemma \ref{4}(2) again.
In summary, we have
\begin{equation}\label{8}
\ext E=\cm(R)\qquad\text{for all }E\in\{M_i,N_i,X_j,Y_j\mid i=1,2\text{ and }j=1,2,3\}.
\end{equation}
Applying Lemma \ref{7} to the exact sequences $\sigma,\tau,\zeta$, we get an exact sequence $0\to C\to R\oplus A\oplus C\to C\to0$, which shows that $R,A\in\ext C$.
It follows from Lemma \ref{4}(1) that $B\in\ext D$, and thus $\add\{B,D\}\subseteq\ext D$.

Let $\p$ be the prime ideal of $R$ generated by $x$.
We claim that $B,D$ are the only nonisomorphic indecomposable maximal Cohen--Macaulay $R$-modules $X$ with $X_\p=0$.
Indeed, by \eqref{8}, for all $E\in\{M_i,N_i,X_j,Y_j\mid i=1,2\text{ and }j=1,2,3\}$ the module $R$ belongs to $\ext E$, so that $R_\p$ belongs to $\ext E_\p$ by Lemma \ref{3}(2), and in particular, $E_\p\ne0$.
Since $A=R/(x)$ and $B=R/(x^2+y^3)$, we have $A_\p\ne0=B_\p$.
There are equivalences
$$
\left(\begin{smallmatrix}x&y\\y^2&-x\end{smallmatrix}\right)
\cong\left(\begin{smallmatrix}x&1\\y^2&-y^{-1}x\end{smallmatrix}\right)
\cong\left(\begin{smallmatrix}0&1\\y^2+y^{-1}x^2&-y^{-1}x\end{smallmatrix}\right)
\cong\left(\begin{smallmatrix}0&1\\y^2+y^{-1}x^2&0\end{smallmatrix}\right)
\cong\left(\begin{smallmatrix}0&1\\1&0\end{smallmatrix}\right)
\cong\left(\begin{smallmatrix}1&0\\0&1\end{smallmatrix}\right)
$$
of matrices over $R_\p$.
Hence it holds that $D_\p\cong\cok\left(\begin{smallmatrix}1&0\\0&1\end{smallmatrix}\right)=0$.
The localized exact sequence $\sigma_\p:0\to C_\p\to(X_3)_\p\to D_\p\to0$ shows that $C_\p\cong(X_3)_\p\ne0$.
Thus the claim follows.

It follows from Lemma \ref{3}(2) and the equality $D_\p=0$ that every $R$-module $E$ that belongs to $\ext D$ is such that $E_\p=0$.
Therefore, the above claim implies that $\add\{B,D\}=\ext D$.
Using Lemma \ref{3}(3), we get $\add\{B,D,R\}=\ext\{D,R\}$.
In summary, we have
\begin{equation}\label{13}
0\subsetneq(\ext B=\add B)\subsetneq(\ext D=\add\{B,D\})\subsetneq(\ext\{D,R\}=\add\{B,D,R\})\subsetneq\cm(R).
\end{equation}
Suppose that there exists an extension-closed subcategory $\X$ with $\add\{B,D,R\}\subsetneq\X\subsetneq\cm(R)$.
Then $\X$ contains some module
$$
E\in\ind\cm(R)\setminus\{B,D,R\}=\{A,C,M_i,N_i,X_j,Y_j\mid i=1,2\text{ and }j=1,2,3\}.
$$
In view of \eqref{8}, the module $E$ must be either $A$ or $C$.
Hence $\X$ contains either $\ext\{A,B\}$ or $\ext\{C,D\}$.
The exact sequences $\rho,\sigma$ give rise to equalities $\ext\{A,B\}=\ext\{C,D\}=\cm(R)$.
Therefore we have $\X=\cm(R)$, which is a contradiction.
Thus, the chain \eqref{13} is saturated.

Applying Lemma \ref{10} to $\X=\{D\}$, we obtain $\ext(\syz\{D\})=\add\syz(\ext D)$.
Hence
$$
\ext C=\ext\{C,R\}=\ext(\syz\{D\})=\add\syz(\ext D)=\add\{\syz B,\syz D,R\}=\add\{A,C,R\}.
$$
We have a chain of subcategories of $\cm(R)$:
\begin{equation}\label{14}
0\subsetneq(\ext A=\add A)\subsetneq(\ext\{A,R\}=\add\{A,R\})\subsetneq(\ext C=\add\{A,C,R\})\subsetneq\cm(R).
\end{equation}
Suppose that there exists an extension-closed subcategory $\X$ with $\add\{A,C,R\}\subsetneq\X\subsetneq\cm(R)$.
Then $\X$ contains some module $E\in\ind\cm(R)\setminus\{A,C,R\}=\{B,D,M_i,N_i,X_j,Y_j\mid i=1,2\text{ and }j=1,2,3\}$.
By \eqref{8} we must have $E\in\{B,D\}$.
Hence $\X$ contains either $\ext\{A,B\}$ or $\ext\{C,D\}$, and get a contradiction as before.
Thus, the chain \eqref{14} is saturated.
Now we obtain the Hasse diagram in the theorem.
\end{proof}

\begin{thm}\label{17}
Let $R=k[\![x,y]\!]/(x^2y+y^{n-1})$ be the $\d_n^1$-singularity, where $n\ge4$ is an integer.
Set $A=R/(y)$, $B=R/(x^2+y^{n-2})$, $M_j=\cok\left(\begin{smallmatrix}x&y^j\\y^{n-j-2}&-x\end{smallmatrix}\right)$ and $N_j=\cok\left(\begin{smallmatrix}xy&y^{j+1}\\y^{n-j-1}&-xy\end{smallmatrix}\right)$ for each $0\le j\le n-3$.
\begin{enumerate}[\rm(1)]
\item
Suppose that $n$ is odd and put $l=\frac{n-3}{2}$.
Then the Hasse diagram of the poset $\E(R)$ is the following.
$$
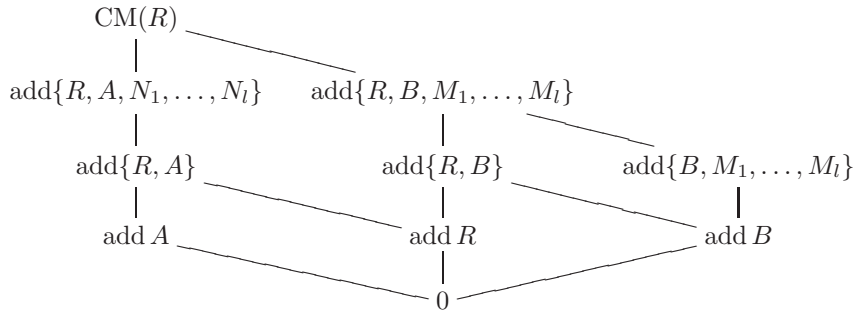
\xymatrix@R-1pc@C-1pc{
\cm(R)\ar@{-}[d]\ar@{-}[rd]\\
\add\{R,A,N_1,\dots,N_l\}\ar@{-}[d]&\add\{R,B,M_1,\dots,M_l\}\ar@{-}[d]\ar@{-}[rd]\\
\add\{R,A\}\ar@{-}[d]\ar@{-}[rd]&\add\{R,B\}\ar@{-}[d]\ar@{-}[rd]&\add\{B,M_1,\dots,M_l\}\ar@{-}[d]\\
\add A\ar@{-}[rd]&\add R\ar@{-}[d]&\add B\ar@{-}[ld]\\
&0
}
$$
\item
Suppose that $n$ is even and put $l=\frac{n-4}{2}$.
Then the Hasse diagram of the poset $\E(R)$ is the following.
$$
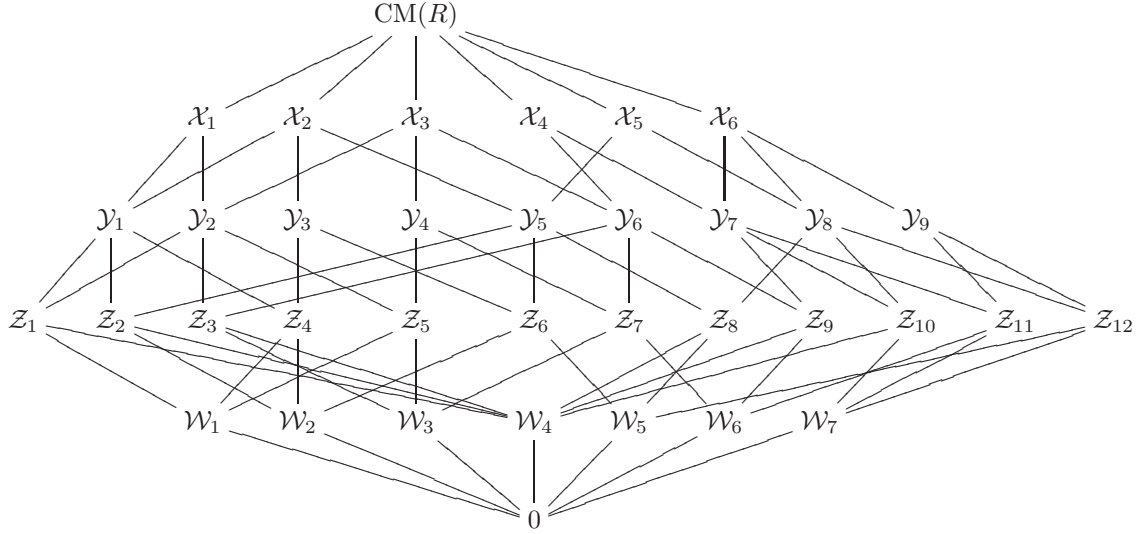
\xymatrix@R2pc@C1.3pc{
&&&&\cm(R)\ar@{-}[lld]\ar@{-}[ld]\ar@{-}[d]\ar@{-}[rd]\ar@{-}[rrd]\ar@{-}[rrrd]\\
&&\X_1\ar@{-}[ld]\ar@{-}[d]&\X_2\ar@{-}[lld]\ar@{-}[d]\ar@{-}[rrd]&\X_3\ar@{-}[lld]\ar@{-}[d]\ar@{-}[rrd]&\X_4\ar@{-}[rd]\ar@{-}[rrd]&\X_5\ar@{-}[ld]\ar@{-}[rrd]&\X_6\ar@{-}[d]\ar@{-}[rd]\ar@{-}[rrd]\\
&\Y_1\ar@{-}[ld]\ar@{-}[d]\ar@{-}[rrd]&\Y_2\ar@{-}[lld]\ar@{-}[d]\ar@{-}[rrd]&\Y_3\ar@{-}[d]\ar@{-}[rrd]&\Y_4\ar@{-}[d]\ar@{-}[rrd]&\Y_5\ar@{-}[lllld]\ar@{-}[d]\ar@{-}[rrd]&\Y_6\ar@{-}[lllld]\ar@{-}[d]\ar@{-}[rrd]&\Y_7\ar@{-}[rd]\ar@{-}[rrd]\ar@{-}[rrrd]&\Y_8\ar@{-}[ld]\ar@{-}[rd]\ar@{-}[rrrd]&\Y_9\ar@{-}[rd]\ar@{-}[rrd]\\
\ZZ_1\ar@{-}[rrd]\ar@{-}[rrrrrd]&\ZZ_2\ar@{-}[rrd]\ar@{-}[rrrrd]&\ZZ_3\ar@{-}[rrd]\ar@{-}[rrrd]&\ZZ_4\ar@{-}[ld]\ar@{-}[d]&\ZZ_5\ar@{-}[lld]\ar@{-}[d]&\ZZ_6\ar@{-}[lld]\ar@{-}[rd]&\ZZ_7\ar@{-}[lld]\ar@{-}[rd]&\ZZ_8\ar@{-}[lld]\ar@{-}[ld]&\ZZ_9\ar@{-}[llld]\ar@{-}[ld]&\ZZ_{10}\ar@{-}[lllld]\ar@{-}[ld]&\ZZ_{11}\ar@{-}[llld]\ar@{-}[lld]&\ZZ_{12}\ar@{-}[llllld]\ar@{-}[llld]\\
&&\W_1\ar@{-}[rrrd]&\W_2\ar@{-}[rrd]&\W_3\ar@{-}[rd]&\W_4\ar@{-}[d]&\W_5\ar@{-}[ld]&\W_6\ar@{-}[lld]&\W_7\ar@{-}[llld]\\
&&&&&0}
$$
The vertices in the diagram are as follows, where $C_\pm=R/(xy\pm\sqrt{-1}\,y^{l+2})$ and $D_\pm=R/(x\mp\sqrt{-1}\,y^{l+1})$.
$$
\begin{array}{l}
\X_1=\add\{R,A,C_+,C_-,N_1,\dots,N_l\},\ 
\X_2=\add\{R,A,C_+,D_-\},\ 
\X_3=\add\{R,A,C_-,D_+\},\\
\X_4=\add\{R,B,C_-,D_+\},\ 
\X_5=\add\{R,B,C_+,D_-\},\ 
\X_6=\add\{R,B,D_+,D_-,M_1,\dots,M_l\},\\
\Y_1=\add\{R,A,C_+\},\ 
\Y_2=\add\{R,A,C_-\},\ 
\Y_3=\add\{A,C_+,D_-\},\ 
\Y_4=\add\{A,C_-,D_+\},\\
\Y_5=\add\{R,C_+,D_-\},\ 
\Y_6=\add\{R,C_-,D_+\},\ 
\Y_7=\add\{R,B,D_+\},\ 
\Y_8=\add\{R,B,D_-\},\\
\Y_9=\add\{B,D_+,D_-,M_1,\dots,M_l\},\ 
\ZZ_1=\add\{R,A\},\ 
\ZZ_2=\add\{R,C_+\},\ 
\ZZ_3=\add\{R,C_-\},\\
\ZZ_4=\add\{A,C_+\},\ 
\ZZ_5=\add\{A,C_-\},\ 
\ZZ_6=\add\{C_+,D_-\},\ 
\ZZ_7=\add\{C_-,D_+\},\ 
\ZZ_8=\add\{R,D_-\},\\
\ZZ_9=\add\{R,D_+\},\ 
\ZZ_{10}=\add\{R,B\},\ 
\ZZ_{11}=\add\{B,D_+\},\ 
\ZZ_{12}=\add\{B,D_-\},\ 
\W_1=\add A,\\
\W_2=\add C_+,\ 
\W_3=\add C_-,\ 
\W_4=\add R,\ 
\W_5=\add D_-,\ 
\W_6=\add D_+,\ 
\W_7=\add B.
\end{array}
$$
\end{enumerate}
\end{thm}

\begin{proof}
Put $X_j=\cok\left(\begin{smallmatrix}x&y^j\\y^{n-j-1}&-xy\end{smallmatrix}\right)$ and $Y_j=\cok\left(\begin{smallmatrix}xy&y^j\\y^{n-j-1}&-x\end{smallmatrix}\right)$ for each $0\le j\le n-3$.
By \cite[(9.11) and (9.12)]{Y} and Lemma \ref{6}, we have that
\begin{equation}\label{20}
\ind\cm(R)\supseteq\{R,A,B,X_{l+1},Y_{l+1},X_j,Y_j,M_j,N_j\mid 1\le j\le l\},
\end{equation}
and that there exist exact sequences
$$
\begin{array}{l}
\rho_j:0\to Y_{j+1}\to M_{j+1}\oplus N_j\to X_{j+1}\to0,\qquad
0\to M_j\to X_j\oplus Y_{j+1}\to N_j\to0,\\
\zeta_j:0\to X_{j+1}\to N_{j+1}\oplus M_j\to Y_{j+1}\to0,\qquad
\sigma_j:0\to N_j\to Y_j\oplus X_{j+1}\to M_j\to0
\end{array}
\quad(0\le j\le l)
$$
where the maps $F\to G$ with $F,G\in\ind\cm(R)$ are all the same, and isomorphisms $M_0\cong B$, $N_0\cong A\oplus R$, $X_0\cong Y_0\cong R$, $\syz A\cong B$, $\syz X_j\cong Y_j\cong(x,y^j)R$ and $\syz N_j\cong M_j\cong(xy,y^{j+1})R$ for all $1\le j\le l+1$.
For each $1\le j\le l$ there are exact squares in the left below, which produce exact sequences in the right below.
$$
\xymatrix@R-1.3pc@C-.2pc{
M_j\ar[r]\ar[d]&X_j\ar[r]\ar[d]&M_{j-1}\ar[r]\ar[d]&X_{j-1}\ar[d]&&0\to M_j\to M_{j+1}\oplus M_{j-1}\to M_j\to0,\\
Y_{j+1}\ar[r]\ar[d]&N_j\ar[r]\ar[d]&Y_j\ar[r]\ar[d]&N_{j-1}\ar[r]\ar[d]&Y_{j-1}\ar[d]&0\to X_j\to X_{j+1}\oplus X_{j-1}\to X_j\to0,\\
M_{j+1}\ar[r]&X_{j+1}\ar[r]\ar[d]&M_j\ar[r]\ar[d]&X_j\ar[r]\ar[d]&M_{j-1}\ar[d]&0\to N_j\to N_{j+1}\oplus N_{j-1}\to N_j\to0,\\
&N_{j+1}\ar[r]&Y_{j+1}\ar[r]&N_j\ar[r]&Y_j&0\to Y_j\to Y_{j+1}\oplus Y_{j-1}\to Y_j\to0.}
$$
These four exact sequences show that
$$
\begin{array}{l}
M_{l+1}\in\ext M_l=\ext M_{l-1}=\cdots=\ext M_1\ni M_0=B,\qquad
\ext X_l=\ext X_{l-1}=\cdots=\ext X_1,\\
N_{l+1}\in\ext N_l=\ext N_{l-1}=\cdots=\ext N_1\ni N_0=A\oplus R,\qquad
\ext Y_l=\ext Y_{l-1}=\cdots=\ext Y_1.
\end{array}
$$
Since $Y_1$ is isomorphic to the maximal ideal $(x,y)R$, we have $\ext Y_1=\cm(R)$ by \cite[Corollary 2.7]{stcm}.
Lemma \ref{4}(2) implies that $\ext X_1=\cm(R)$.
Thus for all $1\le j\le l$ there are inclusions and equalities
\begin{equation}\label{15}
\ext M_j\supseteq\add\{B,M_1,\dots,M_{l+1}\},\quad
\ext N_j\supseteq\add\{R,A,N_1,\dots,N_{l+1}\},\quad
\ext X_j=\ext Y_j=\cm(R).
\end{equation}

(1) Let $n$ be odd.
By Lemmas \ref{2} and \ref{1}, we get the subdiagram of the Hasse diagram of $\E(R)$:
$$
\xymatrix@R-1pc@C+1pc{
\add\{R,A\}\ar@{-}[d]\ar@{-}[rd]&\add\{R,B\}\ar@{-}[d]\ar@{-}[rd]\\
\add A\ar@{-}[rd]&\add R\ar@{-}[d]&\add B\ar@{-}[ld]\\
&0
}
$$
By virtue of \cite[(9.11)]{Y}, the inclusion \eqref{20} is an equality.
Put $\p=(y)\in\spec R$.
It is clear that $(M_0)_\p=B_\p=0\ne A_\p$.
For each $1\le j\le l+1$, as $M_j\cong(xy,y^{j+1})R$ and $Y_j\cong(x,y^j)R$, we have $(M_j)_\p\subseteq\p R_\p=0$ and $(Y_j)_\p\cong R_\p$.
For each $1\le j\le l+1$, there is an exact sequence $0\to X_j\to R^{\oplus2}\to Y_j\to0$, which shows that $(X_j)_\p\cong R_\p$.
For each $1\le j\le l$, the exact sequence $\sigma_j$ shows $(N_j)_\p\ne0$.
Hence there is an equality
\begin{equation}\label{16}
\{E\in\ind\cm(R)\mid E_\p=0\}=\{B,M_1,\dots,M_l\}.
\end{equation}

Since $n$ is odd, we have isomorphisms $X_{l+1}\cong Y_{l+1}$, $M_l\cong M_{l+1}$ and $N_l\cong N_{l+1}$ by \cite[(9.11.5)]{Y}.
The short exact sequence $\sigma_l$ gives rise to a short exact sequence $0\to N_{l+1}\to Y_l\oplus Y_{l+1}\to M_{l+1}\to0$.
Therefore, there are exact squares in the left below, which produce an exact sequence in the right below.
$$
\xymatrix@R-1.3pc@C+1.3pc{
X_{l+1}\ar[r]\ar[d]&M_l\ar[r]\ar[d]&X_l\ar[d]\\
N_{l+1}\ar[r]\ar[d]&Y_{l+1}\ar[r]\ar[d]&N_l\ar[d]&0\to X_{l+1}\to X_l\oplus Y_l\to X_{l+1}\to0.
\\
Y_l\ar[r]&M_{l+1}\ar[r]&X_{l+1}
}
$$
This exact sequence and \eqref{15} give $\cm(R)=\ext X_l\subseteq\ext X_{l+1}$, which and Lemma \ref{4}(2) yield $\cm(R)=\ext X_{l+1}=\ext Y_{l+1}$.
By \eqref{16}, for any $1\le j\le l$ and any $E\in\ext M_j$ it holds that $E_\p=0$.
Using \eqref{15} and \eqref{16}, we get $\ext M_j=\add\{B,M_1,\dots,M_l\}$ for each $1\le j\le l$.
This equality, \eqref{15}, and Lemma \ref{4}(1) yield $\ext N_j=\add\{R,A,N_1,\dots,N_l\}$ for each $1\le j\le l$.
Applying Lemma \ref{3}(3), we get $\ext\{R,M_j\}=\add\{R,B,M_1,\dots,M_l\}$ for every $1\le j\le l$.
The exact sequence $\sigma_l$ shows that $X_{l+1}\in\ext\{M_l,N_l\}$, which implies $\ext\{M_l,N_l\}=\cm(R)$.
Therefore, $\ext\{M_j,N_h\}=\cm(R)$ for all $1\le j,h\le l$.
From $\sigma_0$ and the proof of Lemma \ref{3}(3) (or \cite[page 78, line 2]{Y}) we get an exact sequence $0\to A\to X_1\to B\to0$, which implies $X_1\in\ext\{A,B\}$, and $\ext\{A,B\}=\cm(R)$ by \eqref{15}.
Now we obtain the Hasse diagram as in the theorem.

(2) We consider the case where $n$ is even.
In view of \cite[(9.12)]{Y} and Lemma \ref{6}, we have that
$$
\ind\cm(R)=\{R,A,B,C_\pm,D_\pm,X_{l+1},Y_{l+1},X_j,Y_j,M_j,N_j\mid 1\le j\le l\}
$$
with $\syz C_\pm\cong D_\pm$, $M_{l+1}\cong D_+\oplus D_-$ and $N_{l+1}\cong C_+\oplus C_-$, and that there are exact sequences
$$
\begin{array}{l}
\alpha_\pm:0\to C_\pm\to Y_{l+1}\to D_\pm\to0,\qquad
\beta:0\to B\to Y_1\to A\to0,\\
\gamma_\pm:0\to D_\pm\to X_{l+1}\to C_\pm\to0,\qquad
\delta:0\to A\to X_1\to B\to0.
\end{array}
$$
By Lemma \ref{1} we see that $R,A,B,C_\pm,D_\pm$ are rigid.
As is shown in the proof of \cite[Proposition 2.6]{BIKR}, the direct sums $C_+\oplus D_-,\,C_-\oplus D_+,\,C_+\oplus A,\,D_+\oplus B,\,A\oplus C_-,\,B\oplus D_-$ are (maximal) rigid.
Applying Lemma \ref{2}, we see that the following subcategories are extension-closed.
$$
\Y_1,\Y_2,\Y_5,\Y_6,\Y_7,\Y_8,\ZZ_1,\ZZ_2,\ZZ_3,\ZZ_4,\ZZ_5,\ZZ_6,\ZZ_7,\ZZ_8,\ZZ_9,\ZZ_{10},\ZZ_{11},\ZZ_{12},\W_1,\W_2,\W_3,\W_4,\W_5,\W_6,\W_7.
$$
Taking the direct sum of $\alpha_+$ and $\alpha_-$, we get an exact sequence $0\to N_{l+1}\to Y_{l+1}\oplus Y_{l+1}\to M_{l+1}\to0$.
Thus there are exact squares in the left below, which produce an exact sequence in the right below.
$$
\xymatrix@R-1.3pc@C+1.3pc{
X_{l+1}\ar[r]\ar[d]&M_l\ar[r]\ar[d]&X_l\ar[d]\\
N_{l+1}\ar[r]\ar[d]&Y_{l+1}\ar[r]\ar[d]&N_l\ar[d]&0\to X_{l+1}\to X_l\oplus Y_{l+1}\to X_{l+1}\to0.
\\
Y_{l+1}\ar[r]&M_{l+1}\ar[r]&X_{l+1}
}
$$
Hence $X_l$ belongs to $\ext X_{l+1}$.
Thanks to \eqref{15} and Lemma \ref{4}(2), for all integers $1\le h\le l$, it holds that
$$
\ext Y_h=\ext X_h=\ext Y_1=\ext X_l=\ext X_{l+1}=\ext Y_{l+1}=\cm(R).
$$
By $\beta,\gamma_\pm$ we have $\ext\{A,B\}=\ext\{C_+,D_+\}=\ext\{C_-,D_-\}=\cm(R)$. 
As $M_{l+1}\cong D_+\oplus D_-$, it holds that
$$
\begin{array}{l}
C_+,D_+\in\ext\{C_+,D_+,D_-\}=\ext\{M_{l+1},C_+\}\subseteq\ext\{M_j,C_+\},\\
C_-,D_-\in\ext\{C_-,D_+,D_-\}=\ext\{M_{l+1},C_-\}\subseteq\ext\{M_j,C_-\}
\end{array}
$$
for all $1\le j\le l$, where the inclusions follow from \eqref{15}.
Hence $\ext\{M_j,C_+\}=\ext\{M_j,C_-\}=\cm(R)$ hold for every $1\le j\le l+1$.
Applying Lemma \ref{4}(2), we get $\ext\{N_j,D_+\}=\ext\{N_j,D_-\}=\cm(R)$ for every $1\le j\le l+1$.
It is observed from \eqref{15} that $A,B\in\ext\{A,M_j\}\cap\ext\{B,N_j\}\cap\ext\{M_j,N_{j'}\}$, which implies that $\ext\{A,M_j\}=\ext\{B,N_j\}=\ext\{M_j,N_{j'}\}=\cm(R)$ for all $1\le j,j'\le l$.
So far, we have shown that
\begin{equation}\label{19}
\ext\C=\cm(R)\text{ for all }\C\in\left\{\begin{array}{r|l}
\begin{matrix}
\{X_h\},\{Y_h\},\{A,B\},\{C_+,D_+\},\{C_-,D_-\},\\
\{C_+,M_j\},\{C_-,M_j\},\{D_+,N_j\},\{D_-,N_j\},\\
\{A,M_j\},\{B,N_j\},\{M_j,N_{j'}\}
\end{matrix}&
\begin{matrix}
1\le h\le l+1\\
\,1\le j,j'\le l
\end{matrix}\end{array}\right\}.
\end{equation}
Since $M_{l+1}\cong D_+\oplus D_-$ and $N_0\cong A\oplus R$, the exact sequences $\alpha_\pm,\rho_l,\sigma_l,\rho_{l-1},\sigma_{l-1},\dots,\rho_1,\sigma_1,\rho_0,\delta$ and the exact sequences $\gamma_+,\zeta_l,\alpha_-$ respectively give rise to the following two diagrams of exact squares.
$$
\begin{array}{l}
\xymatrix@R-1pc@C-1.4pc{
C_\pm\ar[r]\ar[d]&Y_{l+1}\ar[r]\ar[d]&D_\mp\oplus N_l\ar[r]\ar[d]&D_\mp\oplus Y_l\ar[r]\ar[d]&D_\mp\oplus N_{l-1}\ar[r]\ar[d]&\cdots\ar[r]&D_\mp\oplus N_1\ar[r]\ar[d]&D_\mp\oplus Y_1\ar[r]\ar[d]&D_\mp\oplus A\oplus R\ar[r]\ar[d]&D_\mp\oplus R\ar[d]\\
0\ar[r]&D_\pm\ar[r]&X_{l+1}\ar[r]&M_l\ar[r]&X_l\ar[r]&\cdots\ar[r]&X_2\ar[r]&M_1\ar[r]&X_1\ar[r]&B
}\\
\xymatrix@R-1pc@C-1pc{
D_+\ar[r]\ar[d]&X_{l+1}\ar[r]\ar[d]&M_l\oplus C_-\ar[r]\ar[d]&M_l\ar[d]\\
0\ar[r]&C_+\ar[r]&Y_{l+1}\ar[r]&D_-
}
\end{array}
$$
We obtain exact sequences $0\to C_\pm\to D_\mp\oplus R\to B\to0$ and $\varepsilon:0\to D_+\to M_l\to D_-\to0$, which show
\begin{equation}\label{18}
R,D_\pm\in\ext\{B,C_\mp\},\qquad
M_l\in\ext\{D_+,D_-\}.
\end{equation}
It follows from Lemma \ref{4}(1) and \eqref{15} that there are containments
\begin{equation}\label{29}
C_\pm\in\ext\{A,D_\mp\}
\quad\text{and}\quad
B,M_j\in\ext M_l\subseteq\ext\{D_+,D_-\}
\quad\text{for all }1\le j\le l+1.
\end{equation}
Put $\p=(y)$, $\q=(x+\sqrt{-1}\,y^{l+1})R$ and $\r=(x-\sqrt{-1}\,y^{l+1})R$; these are the minimal prime ideals of $R$.
For each $E\in\ind\cm(R)$ we denote by $\chi(E)=(\chi_1(E),\chi_2(E),\chi_3(E))$ the triple of the dimensions of the vector spaces $E_\p,E_\q,E_\r$ over the fields $R_\p,R_\q,R_\r$ respectively.
Note that
$$
A=R/\p,\qquad
B=R/\q\r,\qquad
C_+=R/\p\q,\qquad
C_-=R/\p\r,\qquad
D_+=R/\r,\qquad
D_-=R/\q.
$$
We see that $\chi(R)=(1,1,1)$, $\chi(A)=(1,0,0)$, $\chi(B)=(0,1,1)$, $\chi(C_+)=(1,1,0)$, $\chi(C_-)=(1,0,1)$, $\chi(D_+)=(0,0,1)$ and $\chi(D_-)=(0,1,0)$.
Fix $1\le h\le l+1$ and $1\le j\le l$.
As $\syz X_h\cong Y_h\cong(x,y^h)R$ and $X_h$ is generated by two elements, there is an exact sequence $0\to(x,y^h)R\to R^{\oplus2}\to X_h\to0$.
Localizing this at $\p,\q,\r$ shows $\chi(X_h)=\chi(Y_h)=(1,1,1)$.
Since $\syz N_j\cong M_j\cong(xy,y^{j+1})R$, a similar argument shows $\chi(M_j)=(0,1,1)$ and $\chi(N_j)=(2,1,1)$.
In summary, for all $1\le h\le l+1$ and $1\le j\le l$ we have the following table.
\begin{center}
\begin{tabular}{|c||c|c|c|c|c|c|c|c|} \hline
$E$&$R,X_h,Y_h$&$A$&$B,M_j$&$C_+$&$C_-$&$D_+$&$D_-$&$N_j$\\
\hline
$\chi(E)$&$(1,1,1)$&$(1,0,0)$&$(0,1,1)$&$(1,1,0)$&$(1,0,1)$&$(0,0,1)$&$(0,1,0)$&$(2,1,1)$\\
\hline
\end{tabular}
\end{center}
Since the equalities in the left below hold, so do the equalities in the right below by \eqref{29}.
$$
\begin{cases}
\{E\mid\chi_1(E)=0\}=\{B,M_1,\dots,M_l,D_+,D_-\},\\
\{E\mid\chi_2(E)=0\}=\{A,C_-,D_+\},\\
\{E\mid\chi_3(E)=0\}=\{A,C_+,D_-\}.
\end{cases}
\begin{cases}
\ext\{D_+,D_-\}=\add\{B,M_1,\dots,M_l,D_+,D_-\}=\Y_9,\\
\ext\{A,D_+\}=\add\{A,C_-,D_+\}=\Y_4,\\
\ext\{A,D_-\}=\add\{A,C_+,D_-\}=\Y_3.
\end{cases}
$$
Hence $\Y_3,\Y_4,\Y_9$ are extension-closed.
Using Lemmas \ref{3}(3) and \ref{10}, we obtain the following six equalities.
$$
\begin{array}{l}
\begin{cases}
\ext\{R,D_+,D_-\}=\add\{R,D_+,D_-,B,M_1,\dots,M_l\}=\X_6,\\
\ext\{R,A,D_+\}=\add\{R,A,D_+,C_-\}=\X_3,\\
\ext\{R,A,D_-\}=\add\{R,A,D_-,C_+\}=\X_2,
\end{cases}\\
\begin{cases}
\ext\{R,C_+,C_-\}=\add\{R,C_+,C_-,A,N_1,\dots,N_l\}=\X_1,\\
\ext\{R,B,C_+\}=\add\{R,B,C_+,D_-\}=\X_5,\\
\ext\{R,B,C_-\}=\add\{R,B,C_-,D_+\}=\X_4.
\end{cases}
\end{array}
$$
Consequently, the subcategories $\X_1,\X_2,\X_3,\X_4,\X_5,\X_6$ are extension-closed.
It is observed from \eqref{19} that for each integer $1\le p\le6$ there exists no extension-closed subcategory $\C$ of $\cm(R)$ with $\X_p\subsetneq\C\subsetneq\cm(R)$.

We claim that for each $(p,q)\in\{(1,1),(2,1),(7,6),(8,6)\}$ there exists no extension-closed subcategory $\C$ of $\cm(R)$ with $\Y_p\subsetneq\C\subsetneq\X_q$.
Indeed, we have $\ext(\Y_1\cup\{C_-\})=\ext\{R,A,C_+,C_-\}\supseteq\ext\{R,C_+,C_-\}=\X_1$, so that $\ext(\Y_1\cup\{C_-\})=\X_1$.
Also, by \eqref{15}, for each $1\le j\le l$, it holds that
$$
\ext(\Y_1\cup\{N_j\})=\ext\{R,A,C_+,N_j\}\supseteq\ext\{R,C_+,N_{l+1}\}=\ext\{R,C_+,C_+\oplus C_-\}\supseteq\ext\{R,C_+,C_-\}=\X_1.
$$
Hence $\ext(\Y_1\cup\{N_j\})=\X_1$, and thus the claim follows for $(p,q)=(1,1)$.
A similar argument shows the claim for $(p,q)=(2,1)$.
Also, analogously, we can show that $\ext(\Y_7\cup\{D_-\})=\ext(\Y_7\cup\{M_j\})=\X_6$ for $1\le j\le l$, which deduces the claim for $(p,q)=(7,6)$.
The claim for $(p,q)=(8,6)$ is shown similarly.

It holds that $\ext(\ZZ_{11}\cup\{D_-\})=\ext\{B,D_+,D_-\}\supseteq\ext\{D_+,D_-\}=\Y_9$, and therefore $\ext(\ZZ_{11}\cup\{D_-\})=\Y_9$.
By using \eqref{15}, for each integer $1\le j\le l$ we have
$$
\ext(\ZZ_{11}\cup\{M_j\})=\ext\{B,D_+,M_j\}\supseteq\ext\{B,D_+,M_{l+1}\}=\ext\{B,D_+,D_+\oplus D_-\}\supseteq\ext\{D_+,D_-\}=\Y_9,
$$
whence $\ext(\ZZ_{11}\cup\{M_j\})=\Y_9$.
It is observed that there is no extension-closed subcategory $\C$ of $\cm(R)$ with $\ZZ_{11}\subsetneq\C\subsetneq\Y_9$.
By an analogous argument, there is no extension-closed subcategory $\C$ with $\ZZ_{12}\subsetneq\C\subsetneq\Y_9$.

Now we have obtained the Hasse diagram of $\E(R)$ as in the theorem.
\end{proof}

\begin{rem}
In Theorem \ref{17}(2), we can prove more.
Applying $\syz$ to the exact sequence $\varepsilon$, we get $N_l\in\ext\{C_+,C_-\}$.
It follows from \eqref{15} that $R\in\ext N_l\subseteq\ext\{C_+,C_-\}$.
Using \eqref{18}, we obtain equalities
$$
\X_1=\ext\{C_+,C_-\},\qquad
\X_5=\ext\{B,C_+\},\qquad
\X_4=\ext\{B,C_-\}.
$$
Using \eqref{15}, the isomorphisms $M_{l+1}\cong D_+\oplus D_-$, $N_{l+1}\cong C_+\oplus C_-$ and Lemma \ref{3}(3), for all $1\le j\le l$ we have
$$
\begin{array}{l}
\ext M_j=\add\{B,M_1,\dots,M_{l+1}\}=\add\{B,M_1,\dots,M_l,D_+,D_-\}=\ext\{D_+,D_-\}=\Y_9,\\
\ext N_j=\add\{R,A,N_1,\dots,N_{l+1}\}=\add\{R,A,N_1,\dots,N_l,C_+,C_-\}=\ext\{C_+,C_-\}=\X_1,\\
\ext\{R,M_j\}=\add\{R,B,M_1,\dots,M_l,D_+,D_-\}=\ext\{R,D_+,D_-\}=\X_6.
\end{array}
$$
\end{rem}

\begin{thm}\label{21}
Let $R$ be one of the following hypersurfaces, where $n$ is an integer.
$$
\begin{cases}
\a_n^0:\ k[\![x]\!]/(x^{n+1})&(n\ge1),\\
\a_n^2:\ k[\![x,y,z]\!]/(x^2+y^{n+1}+z^2)&(n\ge1),\\
\d_n^2:\ k[\![x,y,z]\!]/(x^2y+y^{n-1}+z^2)&(n\ge4),\\
\e_6^2:\ k[\![x,y,z]\!]/(x^3+y^4+z^2),\\
\e_7^2:\ k[\![x,y,z]\!]/(x^3+xy^3+z^2),\\
\e_8^2:\ k[\![x,y,z]\!]/(x^3+y^5+z^2).
\end{cases}
$$
Then the Hasse diagram of the poset $\E(R)$ is the following graph.
$$
\xymatrix@R-1pc@C-1pc{
\cm(R)\ar@{-}[d]\\
\add R\ar@{-}[d]\\
0
}
$$
\end{thm}

\begin{proof}
The case $\a_n^0$ follows from \cite[Proposition 5.6]{stcm}.
In the other cases, since $\dim R=2$, Lemma \ref{6} implies that the Auslander--Reiten translation functor is isomorphic to the identity functor.
From the Auslander--Reiten quivers exhibited in \cite[(10.15)]{Y} it is easy to observe that for any nonfree indecomposable maximal Cohen--Macaulay $R$-module $E$ the equality $\ext E=\cm(R)$ holds.
The assertion follows from this.
\end{proof}

\section{On the $\a_\infty,\d_\infty$-singularities with dimension at most two}

In this section, we give a complete classification of the extension-closed subcategories of $\cm_0(R)$ in the case where $R$ is a $P$-singularity with $P\in\{\a_\infty,\d_\infty\}$ and has dimension at most two.

\begin{thm}\label{26}
Let $R$ be one of the following hypersurfaces.
$$
\begin{cases}
\a_\infty^1:\ k[\![x,y]\!]/(x^2),\\
\d_\infty^2:\ k[\![x,y,z]\!]/(x^2y+z^2).
\end{cases}
$$
Then the Hasse diagram of the poset $\E_0(R)$ is the following.
$$
\xymatrix@R-1pc@C-1pc{
\cm_0(R)\ar@{-}[d]\\
\add R\ar@{-}[d]\\
0
}
$$
\end{thm}

\begin{proof}
Let $E$ be a nonfree indecomposable maximal Cohen--Macaulay $R$-module that is locally free on the punctured spectrum of $R$.
Then there is an isomorphism $\tau E\cong E$; this follows from \cite[(6.1)]{S} in case $\a_\infty^1$, and from Lemma \ref{6} in case $\d_\infty^2$.
From the Auslander--Reiten quiver in \cite[(6.1)]{S} in case $\a_\infty^1$ and in \cite[(6.2)]{S} in case $\d_\infty^2$, it is easy to observe that $\ext E=\cm_0(R)$.
The assertion follows from this.
\end{proof}

\begin{thm}\label{30}
Let $R=k[\![x,y,z]\!]/(xy)$ be the $\a_\infty^2$-singularity.
For each integer $i>0$, let $I_i=(x,z^i)$ and $J_i=(y,z^i)$ be ideals of $R$.
Then the Hasse diagram of the poset $\E_0(R)$ is the following.
$$
\xymatrix@R-1pc@C-1pc{
\cm_0(R)\ar@{-}[d]\ar@{-}[rd]\\
\add\{R,I_1,I_2,\dots\}\ar@{-}[rd]&\add\{R,J_1,J_2,\dots\}\ar@{-}[d]\\
&\add R\ar@{-}[d]\\
&0
}
$$
\end{thm}

\begin{proof}
By Lemma \ref{6} we have $\tau E\cong E$ for every $E\in\ind\cm_0(R)\setminus\{R\}$.
From the Auslander--Reiten quiver given in \cite[(6.2)]{S}, we see that $\ind\cm_0(R)=\{R,I_i,J_i\mid i\in\Z_{>0}\}$ and that
$$
R\in\ext I_1=\ext I_2=\ext I_3=\cdots,\qquad
R\in\ext J_1=\ext J_2=\ext J_3=\cdots.
$$
Hence $\ext\{I_p,J_q\}=\cm_0(R)$ for all integers $p,q>0$.
To get the Hasse diagram in the assertion, it suffices to show that $J_q\notin\ext I_p$ (and then by symmetry we get $I_q\notin\ext J_p$)  for all integers $p,q>0$.
We have only to verify $J_1\notin\ext I_1$, as if $J_q\in\ext I_p$, then $J_1\in\ext J_1=\ext J_q\subseteq\ext I_p=\ext I_1$.

Now, assume that $J_1\in\ext I_1$.
It follows from \cite[Propositions 2.2(1) and 2.4]{tes} that there exist an $R$-module $M$ and a filtration $0=M_0\subsetneq M_1\subsetneq\cdots\subsetneq M_n=M$ of $R$-submodules of $M$ with $M_i/M_{i-1}\cong I_1$ for each $1\le i\le n$ such that $J_1$ is a direct summand of $M$.
We establishes two claims.

\begin{claim}\label{34}
Let $N$ be an $R$-module.
Denote by $R^\times$ the set of units of $R$.
Put $A_j=\left(
\begin{smallmatrix}
y&-z^j\\
0&x
\end{smallmatrix}
\right)$ for each $j>0$.
\begin{enumerate}[(1)]
\item
The following are equivalent for all positive integers $r$ and $l_1,\dots,l_r$.
\begin{enumerate}[(a)]
\item
There exists an exact sequence $0\to I_{l_1}\oplus\cdots\oplus I_{l_r}\to N\to I_1\to0$.
\item
The module $N$ is isomorphic to the cokernel of the $R$-linear map given by the matrix $D(u_1,\dots,u_r)=\left(
\begin{smallmatrix}
A_{l_1}&O&\cdots&O&U_1\\
O&A_{l_2}&\cdots&O&U_2\\
\cdots&\cdots&\cdots&\cdots&\cdots\\
O&O&\cdots&A_{l_r}&U_r\\
O&O&\cdots&O&A_1
\end{smallmatrix}
\right)$, where $U_i=\left(
\begin{smallmatrix}
0&u_i\\
0&0
\end{smallmatrix}
\right)$ with $u_i\in R^\times\cup\{0\}$ for each $1\le i\le r$.
\end{enumerate}
\item
Assume that condition (1b) is satisfied.
\begin{enumerate}[(a)]
\item
If $u_j=0$ for some $1\le j\le r$, then there exist an exact sequence $0\to I_{l_1}\oplus\cdots\oplus I_{l_{j-1}}\oplus I_{l_{j+1}}\oplus\cdots\oplus I_{l_r}\to N'\to I_1\to0$ and an isomorphism $N\cong I_{l_j}\oplus N'$.
\item
If $l_j\le l_h$ and $u_h\in R^\times$ for some $1\le j\ne h\le r$, then $D(u_1,\dots,u_r)\cong D(u_1,\dots,u_{j-1},0,u_{j+1},\dots,u_r)$.
\item
If $r=1$ and $u_1\in R^\times$, then there exists an isomorphism $N\cong I_{l_1+1}\oplus R$.
\item
The module $N$ is isomorphic to either $I_{l_1}\oplus\cdots\oplus I_{l_r}\oplus I_1$ or $I_{l_1}\oplus\cdots\oplus I_{l_{s-1}}\oplus I_{l_s+1}\oplus I_{l_{s+1}}\oplus I_{l_r}\oplus R$ for some $1\le s\le r$.
\end{enumerate}
\end{enumerate}
\end{claim}

\begin{proof}[Proof of Claim]
(1)
(a)\,$\Rightarrow$\,(b):
Applying the horseshoe lemma repeatedly gives rise to a commutative diagram
\begin{equation}\label{36}
\xymatrix@R-1pc@C+3.5pc{
&0\ar[d]&0\ar[d]&0\ar[d]&0\ar[d]&0\ar[d]\\
0&\bigoplus_{i=1}^rI_{l_i}\ar[l]\ar[d]&F_0\ar[l]_-{
\bigoplus_{i=1}^r\left(
\begin{smallmatrix}
x&z^{l_i}
\end{smallmatrix}
\right)
}\ar[d]&F_1\ar[l]_{\bigoplus_{i=1}^rA_{l_i}}\ar[d]&F_2\ar[l]_{\bigoplus_{i=1}^rB_{l_i}}\ar[d]&F_3\ar[l]_{\bigoplus_{i=1}^rA_{l_i}}\ar[d]\\
0&N\ar[l]\ar[d]&E_0\ar[l]\ar[d]&E_1\ar[l]_H\ar[d]&E_2\ar[l]\ar[d]&E_3\ar[l]\ar[d]\\
0&I_1\ar[l]\ar[d]&G_0\ar[l]_{
\left(
\begin{smallmatrix}
x&z
\end{smallmatrix}
\right)
}\ar[d]&G_1\ar[l]_{A_1}\ar[d]&G_2\ar[l]_{B_1}\ar[d]&G_3\ar[l]_{A_1}\ar[d]\\
&0&0&0&0&0
}
\end{equation}
with exact rows and columns, where $F_h=R^{\oplus2r}$, $G_h=R^{\oplus2}$, $E_h=F_h\oplus G_h=R^{\oplus(2r+2)}$ for $0\le h\le3$, $B_j=\left(
\begin{smallmatrix}
x&z^j\\
0&y
\end{smallmatrix}
\right)$ for $j>0$, and $H=\left(
\begin{smallmatrix}
A_{l_1}&O&\cdots&O&C_1\\
O&A_{l_2}&\cdots&O&C_2\\
\cdots&\cdots&\cdots&\cdots&\cdots\\
O&O&\cdots&A_{l_r}&C_r\\
O&O&\cdots&O&A_1
\end{smallmatrix}
\right)$ with $C_i=\left(
\begin{smallmatrix}
a_i&b_i\\
c_i&d_i
\end{smallmatrix}
\right)$ and $a_i,b_i,c_i,d_i\in R$ for $1\le i\le r$.
We do diagram chasing.
Take any vector $v\in G_2$.
The vector $B_1v\in G_1$ comes from $\binom{\zv}{B_1v}\in E_1=F_1\oplus G_1$.
Since $A_1B_1v=\zv$, the map $H$ sends $\binom{\zv}{B_1v}$ to $\left(\begin{smallmatrix}C_1B_1v\\
\cdots\\
C_rB_1v\\
\zv\end{smallmatrix}\right)\in E_0=F_0\oplus G_0$, which comes from $\left(\begin{smallmatrix}C_1B_1v\\
\cdots\\
C_rB_1v
\end{smallmatrix}\right)\in F_0$.
This goes to $\zv\in\bigoplus_{i=1}^rI_{l_i}$ by the map $\bigoplus_{i=1}^r\left(
\begin{smallmatrix}
x&z^{l_i}
\end{smallmatrix}
\right)$ by the snake lemma, so that it belongs to the image of $\bigoplus_{i=1}^rA_{l_i}$, which coincides with the kernel of $\bigoplus_{i=1}^rB_{l_i}$.
Hence $\zv=\left(\begin{smallmatrix}B_{l_1}&&\\
&\cdots&\\
&&B_{l_r}\end{smallmatrix}\right)\left(\begin{smallmatrix}C_1B_1v\\
\cdots\\
C_rB_1v
\end{smallmatrix}\right)=\left(\begin{smallmatrix}B_{l_1}C_1B_1v\\
\cdots\\
B_{l_r}C_rB_1v
\end{smallmatrix}\right)$.
Since this holds for any vector $v\in G_2$, it is observed that $B_{l_i}C_iB_1$ is a zero matrix for every $1\le i\le r$.
Thus,
\begin{equation}\label{35}
\left(
\begin{smallmatrix}
0&0\\
0&0\end{smallmatrix}
\right)=\left(
\begin{smallmatrix}
x&z^{l_i}\\
0&y
\end{smallmatrix}
\right)\left(
\begin{smallmatrix}
a_i&b_i\\
c_i&d_i
\end{smallmatrix}
\right)\left(
\begin{smallmatrix}
x&z\\
0&y
\end{smallmatrix}
\right)=\left(
\begin{smallmatrix}
x(xa_i+z^{l_i}c_i)&z(xa_i+z^{l_i}c_i+yz^{l_i-1}d_i)\\
0&y(zc_i+yd_i)
\end{smallmatrix}
\right)\quad\text{for every integer $1\le i\le r$}.
\end{equation}
Fix $1\le i\le r$.
We see that $xa_i+z^{l_i}c_i+yz^{l_i-1}d_i=0$, $xa_i+z^{l_i}c_i=ye_i$ and $zc_i+yd_i=xf_i$ for some $e_i,f_i\in R$.
Note that the sequence
$
R\xleftarrow{\mu=\left(\begin{smallmatrix}x&y&z\end{smallmatrix}\right)}R^{\oplus3}\xleftarrow{
\nu=\left(
\begin{smallmatrix}
y&0&z&0\\
0&x&0&z\\
0&0&-x&-y
\end{smallmatrix}
\right)
}R^{\oplus4}
$
is exact; it is part of a minimal free resolution of the residue field $k$.
We have $xf_i+y(-d_i)+z(-c_i)=0$, so that $\left(\begin{smallmatrix}
f_i\\
-d_i\\
-c_i
\end{smallmatrix}\right)\in R^{\oplus3}$ belongs to the kernel of the above map $\mu$, which is equal to the image of the above map $\nu$.
Hence there exist elements $q_i,s_i,t_i,p_i\in R$ such that
$
\left(\begin{smallmatrix}
f_i\\
-d_i\\
-c_i
\end{smallmatrix}\right)=\left(
\begin{smallmatrix}
y&0&z&0\\
0&x&0&z\\
0&0&-x&-y
\end{smallmatrix}
\right)\left(\begin{smallmatrix}
q_i\\
s_i\\
t_i\\
p_i
\end{smallmatrix}\right)
=\left(\begin{smallmatrix}
yq_i+zt_i\\
xs_i+zp_i\\
-xt_i-yp_i
\end{smallmatrix}\right).
$
We thus get $c_i=xt_i+yp_i$ and $d_i=-xs_i-zp_i$.
As $xa_i+z^{l_i}(xt_i+yp_i)=ye_i$, we have $x(a_i+z^{l_i}t_i)=y(e_i-z^{l_i}p_i)\in(x)\cap(y)=0$.
Hence $a_i+z^{l_i}t_i=yg_i$ for some $g_i\in R$, so that $a_i=yg_i-z^{l_i}t_i$.
Write $b_i=u_i+x\alpha_i+y\beta_i+z\gamma_i$ with $u_i\in R^\times\cup\{0\}$ and $\alpha_i,\beta_i,\gamma_i\in R$.
It follows that
$$
\begin{array}{l}
H=\left(
\begin{smallmatrix}
A_{l_1}&&&C_1\\
&\cdots&&\cdots\\
&&A_{l_r}&C_r\\
&&&A_1
\end{smallmatrix}
\right)=\left(
\begin{smallmatrix}
y&-z^{l_1}&&&&a_1&b_1\\
0&x&&&&c_1&d_1\\
&&\cdots&&&\cdots&\cdots\\
&&&y&-z^{l_r}&a_r&b_r\\
&&&0&x&c_r&d_r\\
&&&&&y&-z\\
&&&&&0&x
\end{smallmatrix}
\right)=\left(
\begin{smallmatrix}
y&-z^{l_1}&&&&yg_1-z^{l_1}t_1&b_1\\
0&x&&&&xt_1+yp_1&-xs_1-zp_1\\
&&\cdots&&&\cdots&\cdots\\
&&&y&-z^{l_r}&yg_r-z^{l_r}t_r&b_r\\
&&&0&x&xt_r+yp_r&-xs_r-zp_r\\
&&&&&y&-z\\
&&&&&0&x
\end{smallmatrix}
\right)\\
\phantom{H}\cong\left(
\begin{smallmatrix}
y&-z^{l_1}&&&&-z^{l_1}t_1&b_1\\
0&x&&&&xt_1+yp_1&-xs_1-zp_1\\
&&\cdots&&&\cdots&\cdots\\
&&&y&-z^{l_r}&-z^{l_r}t_r&b_r\\
&&&0&x&xt_r+yp_r&-xs_r-zp_r\\
&&&&&y&-z\\
&&&&&0&x
\end{smallmatrix}
\right)\cong\left(
\begin{smallmatrix}
y&-z^{l_1}&&&&-z^{l_1}t_1&b_1\\
0&x&&&&xt_1+yp_1&-zp_1\\
&&\cdots&&&\cdots&\cdots\\
&&&y&-z^{l_r}&-z^{l_r}t_r&b_r\\
&&&0&x&xt_r+yp_r&-zp_r\\
&&&&&y&-z\\
&&&&&0&x
\end{smallmatrix}
\right)\\
\phantom{H}\cong\left(
\begin{smallmatrix}
y&-z^{l_1}&&&&0&b_1\\
0&x&&&&yp_1&-zp_1\\
&&\cdots&&&\cdots&\cdots\\
&&&y&-z^{l_r}&0&b_r\\
&&&0&x&yp_r&-zp_r\\
&&&&&y&-z\\
&&&&&0&x
\end{smallmatrix}
\right)\cong\left(
\begin{smallmatrix}
y&-z^{l_1}&&&&0&b_1\\
0&x&&&&0&0\\
&&\cdots&&&\cdots&\cdots\\
&&&y&-z^{l_r}&0&b_r\\
&&&0&x&0&0\\
&&&&&y&-z\\
&&&&&0&x
\end{smallmatrix}
\right)=\left(
\begin{smallmatrix}
y&-z^{l_1}&&&&0&u_1+x\alpha_1+y\beta_1+z\gamma_1\\
0&x&&&&0&0\\
&&\cdots&&&\cdots&\cdots\\
&&&y&-z^{l_r}&0&u_r+x\alpha_r+y\beta_r+z\gamma_r\\
&&&0&x&0&0\\
&&&&&y&-z\\
&&&&&0&x
\end{smallmatrix}
\right)\\
\phantom{H}\cong\left(
\begin{smallmatrix}
y&-z^{l_1}&&&&0&u_1+z\gamma_1\\
0&x&&&&0&0\\
&&\cdots&&&\cdots&\cdots\\
&&&y&-z^{l_r}&0&u_r+z\gamma_r\\
&&&0&x&0&0\\
&&&&&y&-z\\
&&&&&0&x
\end{smallmatrix}
\right)\cong\left(
\begin{smallmatrix}
y&-z^{l_1}&&&&y\gamma_1&u_1\\
0&x&&&&0&0\\
&&\cdots&&&\cdots&\cdots\\
&&&y&-z^{l_r}&y\gamma_r&u_r\\
&&&0&x&0&0\\
&&&&&y&-z\\
&&&&&0&x
\end{smallmatrix}
\right)\cong\left(
\begin{smallmatrix}
y&-z^{l_1}&&&&0&u_1\\
0&x&&&&0&0\\
&&\cdots&&&\cdots&\cdots\\
&&&y&-z^{l_r}&0&u_r\\
&&&0&x&0&0\\
&&&&&y&-z\\
&&&&&0&x
\end{smallmatrix}
\right).
\end{array}
$$

(b)\,$\Rightarrow$\,(a):
For each integer $1\le i\le r$, let $a_i=c_i=d_i=0$ and $b_i=u_i\in R^\times\cup\{0\}$.
We easily verify that the first equality in \eqref{35} holds.
This means that \eqref{36} is a commutative diagram with exact rows and columns if we put $C_i=\left(
\begin{smallmatrix}
a_i&b_i\\
c_i&d_i
\end{smallmatrix}
\right)=\left(
\begin{smallmatrix}
0&u_i\\
0&0
\end{smallmatrix}
\right)=U_i$ for every integer $1\le i\le r$.

(2)(a) We may assume $j=1$.
Then $D(u_1,\dots,u_r)=D(0,u_2,\dots,u_r)=A_{l_1}\oplus D(u_2,\dots,u_r)$.
Set $N'=\cok D(u_2,\dots,u_r)$.
We have $N\cong I_{l_1}\oplus N'$.
By (1) there is an exact sequence $0\to I_{l_2}\oplus\cdots\oplus I_{l_r}\to N'\to I_1\to0$.

(b) Replacing $A_{l_1},\dots,A_{l_r}$ if necessary, we may assume $j=1$ and $h=2$, so $l_2-l_1\ge0$.
We have
$$
\begin{array}{l}
D(u_1,u_2,\dots,u_r)=\left(
\begin{smallmatrix}
y&-z^{l_1}&0&0&\cdots&0&u_1\\
0&x&0&0&\cdots&0&0\\
0&0&y&-z^{l_2}&\cdots&0&u_2\\
0&0&0&x&\cdots&0&0\\
\cdots&\cdots&\cdots&\cdots&\cdots&\cdots&\cdots
\end{smallmatrix}
\right)\cong\left(
\begin{smallmatrix}
y&-z^{l_1}&-u_2^{-1}u_1y&u_2^{-1}u_1z^{l_2}&\cdots&0&0\\
0&x&0&0&\cdots&0&0\\
0&0&y&-z^{l_2}&\cdots&0&u_2\\
0&0&0&x&\cdots&0&0\\
\cdots&\cdots&\cdots&\cdots&\cdots&\cdots&\cdots
\end{smallmatrix}
\right)\\
\phantom{D(u_1,u_2,\dots,u_r)}
\cong\left(
\begin{smallmatrix}
y&-z^{l_1}&0&u_2^{-1}u_1z^{l_2}&\cdots&0&0\\
0&x&0&0&\cdots&0&0\\
0&0&y&-z^{l_2}&\cdots&0&u_2\\
0&0&0&x&\cdots&0&0\\
\cdots&\cdots&\cdots&\cdots&\cdots&\cdots&\cdots
\end{smallmatrix}
\right)\cong\left(
\begin{smallmatrix}
y&-z^{l_1}&0&0&\cdots&0&0\\
0&x&0&u_2^{-1}u_1z^{l_2-l_1}x&\cdots&0&0\\
0&0&y&-z^{l_2}&\cdots&0&u_2\\
0&0&0&x&\cdots&0&0\\
\cdots&\cdots&\cdots&\cdots&\cdots&\cdots&\cdots
\end{smallmatrix}
\right)\\
\phantom{D(u_1,u_2,\dots,u_r)}
\cong\left(
\begin{smallmatrix}
y&-z^{l_1}&0&0&\cdots&0&0\\
0&x&0&0&\cdots&0&0\\
0&0&y&-z^{l_2}&\cdots&0&u_2\\
0&0&0&x&\cdots&0&0\\
\cdots&\cdots&\cdots&\cdots&\cdots&\cdots&\cdots
\end{smallmatrix}
\right)=D(0,u_2,\dots,u_r).
\end{array}
$$

(c) Put $p=l_1$ and $h=u_1$.
Since $h$ is a unit of $R$, we have
$$
\begin{array}{l}
D(u_1)=D(h)=\left(
\begin{smallmatrix}
y&-z^p&0&h\\
0&x&0&0\\
0&0&y&-z\\
0&0&0&x
\end{smallmatrix}
\right)\cong\left(
\begin{smallmatrix}
0&0&0&h\\
0&x&0&0\\
h^{-1}yz&-h^{-1}z^{p+1}&y&-z\\
0&h^{-1}xz^p&0&x
\end{smallmatrix}
\right)\cong\left(
\begin{smallmatrix}
0&0&0&1\\
0&x&0&0\\
h^{-1}yz&-h^{-1}z^{p+1}&y&0\\
0&h^{-1}xz^p&0&0
\end{smallmatrix}
\right)\\
\phantom{D(u_1)=D(h)=\left(
\begin{smallmatrix}
y&-z^p&0&h\\
0&x&0&0\\
0&0&y&-z\\
0&0&0&x
\end{smallmatrix}
\right)}
\cong\left(
\begin{smallmatrix}
0&0&0&1\\
0&x&0&0\\
0&-h^{-1}z^{p+1}&y&0\\
0&0&0&0
\end{smallmatrix}
\right)\cong\left(
\begin{smallmatrix}
y&-z^{p+1}&0&0\\
0&x&0&0\\
0&0&0&0\\
0&0&0&1
\end{smallmatrix}
\right).
\end{array}
$$
Therefore, we obtain a desired isomorphism $N\cong I_{p+1}\oplus R$.

(d) The assertion follows from iterated application of (a), (b), (c) and (1).
\renewcommand{\qedsymbol}{$\square$}
\end{proof}

\begin{claim}\label{33}
For every integer $1\le q\le n$ one has $M_q\cong I_{l_1}\oplus\cdots\oplus I_{l_q}$, where $0\le l_1,\dots,l_q\le q$ and $I_0:=R$.
\end{claim}

\begin{proof}[Proof of Claim]
We use induction on $q$.
There is an isomorphism $M_1\cong I_1$, so we are done by putting $l_1=1$ when $q=1$.
We consider the case $q\ge2$.
There exists an exact sequence $0\to M_{q-1}\to M_q\to I_1\to0$.
The induction hypothesis implies $M_{q-1}\cong R^{\oplus e}\oplus I_{l_1}\oplus\cdots\oplus I_{l_r}$ with $e\ge0$, $r=q-1-e$ and $1\le l_1,\dots,l_r\le q-1$.
The same argument as in the proof of Lemma \ref{3}(3) shows that there is an exact sequence $0\to L\to K\to I_1\to0$ such that $L=I_{l_1}\oplus\cdots\oplus I_{l_r}$ and $R^{\oplus e}\oplus K\cong M_q$.
It is observed from Claim \ref{33}(2d) that $K\cong I_{h_1}\oplus\cdots\oplus I_{h_{r+1}}$ for some $0\le h_1,\dots,h_{r+1}\le q$.
Now the assertion follows.
\renewcommand{\qedsymbol}{$\square$}
\end{proof}

Claim \ref{33} yields an isomorphism $M=M_n\cong I_{l_1}\oplus\cdots\oplus I_{l_n}$ with $0\le l_1,\dots,l_n\le n$.
This is a contrariction, since $J_1$ is not a direct summand of this direct sum; note that $J_1,R,I_{l_1},\dots,I_{l_n}$ are all indecomposable $R$-modules.
This contradiction shows that $J_1$ is not in $\ext I_1$, and now the proof of the theorem is completed.
\end{proof}

\begin{thm}\label{27}
Let $R=k[\![x,y]\!]/(x^2y)$ be the $\d_\infty^1$-singularity.
Put $A=R/(y)$ and $B=R/(x^2)$.
Moreover, set $M_j=\cok\left(\begin{smallmatrix}x&y^j\\0&-x\end{smallmatrix}\right)$ and $N_j=\cok\left(\begin{smallmatrix}xy&y^{j+1}\\0&-xy\end{smallmatrix}\right)$ for each integer $j\ge0$.
Then the Hasse diagram of $\E_0(R)$ is:
$$
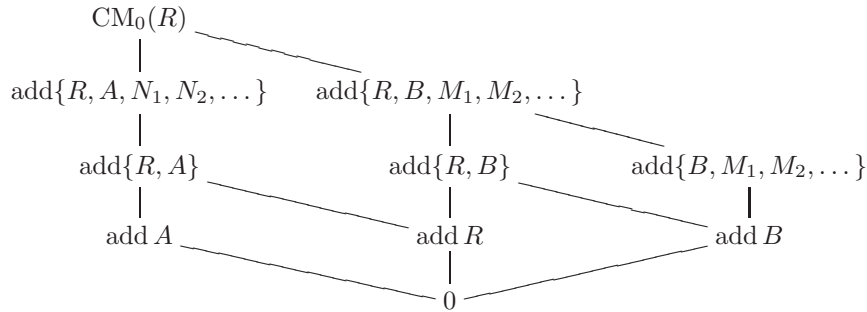
\xymatrix@R-1pc@C-1pc{
\cm_0(R)\ar@{-}[d]\ar@{-}[rd]\\
\add\{R,A,N_1,N_2,\dots\}\ar@{-}[d]&\add\{R,B,M_1,M_2,\dots\}\ar@{-}[d]\ar@{-}[rd]\\
\add\{R,A\}\ar@{-}[d]\ar@{-}[rd]&\add\{R,B\}\ar@{-}[d]\ar@{-}[rd]&\add\{B,M_1,M_2,\dots\}\ar@{-}[d]\\
\add A\ar@{-}[rd]&\add R\ar@{-}[d]&\add B\ar@{-}[ld]\\
&0
}
$$
\end{thm}

\begin{proof}
The proof goes along the same lines as in the proof of Theorem \ref{17}(1).
By Lemmas \ref{2} and \ref{1}, we get the subdiagram of the Hasse diagram of $\E_0(R)$:
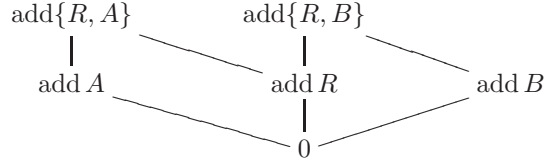
$$
\xymatrix@R-1pc@C+1pc{
\add\{R,A\}\ar@{-}[d]\ar@{-}[rd]&\add\{R,B\}\ar@{-}[d]\ar@{-}[rd]\\
\add A\ar@{-}[rd]&\add R\ar@{-}[d]&\add B\ar@{-}[ld]\\
&0
}
$$
Put $X_j=\cok\left(\begin{smallmatrix}x&y^j\\0&-xy\end{smallmatrix}\right)$ and $Y_j=\cok\left(\begin{smallmatrix}xy&y^j\\0&-x\end{smallmatrix}\right)$ for each integer $j\ge0$.
By \cite[(6.1)]{S} and Lemma \ref{6}, we have that $\ind\cm_0(R)=\{R,A,B,X_j,Y_j,M_j,N_j\mid j\in\Z_{>0}\}$ and there exist exact sequences
$$
\begin{array}{l}
0\to Y_{j+1}\to M_{j+1}\oplus N_j\to X_{j+1}\to0,\qquad
0\to M_j\to X_j\oplus Y_{j+1}\to N_j\to0,\\
0\to X_{j+1}\to N_{j+1}\oplus M_j\to Y_{j+1}\to0,\qquad
\sigma_j:0\to N_j\to Y_j\oplus X_{j+1}\to M_j\to0
\end{array}
\quad(j\in\N),
$$
where the maps $F\to G$ with $F,G\in\ind\cm_0(R)$ are all the same, and isomorphisms $M_0\cong B$, $N_0\cong A\oplus R$, $X_0\cong Y_0\cong R$, $\syz A\cong B$, $\syz X_j\cong Y_j\cong(x,y^j)R$ and $\syz N_j\cong M_j\cong(xy,y^{j+1})R$ for all $j\in\Z_{>0}$.
For each $j\in\Z_{>0}$ there are exact squares in the left below, which produce exact sequences in the right below.
$$
\xymatrix@R-1.3pc@C-.2pc{
M_j\ar[r]\ar[d]&X_j\ar[r]\ar[d]&M_{j-1}\ar[r]\ar[d]&X_{j-1}\ar[d]&&0\to M_j\to M_{j+1}\oplus M_{j-1}\to M_j\to0,\\
Y_{j+1}\ar[r]\ar[d]&N_j\ar[r]\ar[d]&Y_j\ar[r]\ar[d]&N_{j-1}\ar[r]\ar[d]&Y_{j-1}\ar[d]&0\to X_j\to X_{j+1}\oplus X_{j-1}\to X_j\to0,\\
M_{j+1}\ar[r]&X_{j+1}\ar[r]\ar[d]&M_j\ar[r]\ar[d]&X_j\ar[r]\ar[d]&M_{j-1}\ar[d]&0\to N_j\to N_{j+1}\oplus N_{j-1}\to N_j\to0,\\
&N_{j+1}\ar[r]&Y_{j+1}\ar[r]&N_j\ar[r]&Y_j&0\to Y_j\to Y_{j+1}\oplus Y_{j-1}\to Y_j\to0.}
$$
These four exact sequences show that
$$
\begin{array}{l}
B=M_0\in\ext M_1=\ext M_2=\ext M_3=\cdots,\qquad
\ext X_1=\ext X_2=\ext X_3=\cdots,\\
A\oplus R=N_0\in\ext N_1=\ext N_2=\ext N_3=\cdots,\qquad
\ext Y_1=\ext Y_2=\ext Y_3=\cdots.
\end{array}
$$
Since $Y_1$ is isomorphic to the maximal ideal $(x,y)R$, we have $\ext Y_1=\cm_0(R)$ by \cite[Corollary 2.6]{stcm}.
Lemma \ref{4}(2) implies that $\ext X_1=\cm_0(R)$.
Thus for all $j\in\Z_{>0}$ there are inclusions and equalities
$$
\ext M_j\supseteq\add\{B,M_1,M_2,\dots\},\quad
\ext N_j\supseteq\add\{R,A,N_1,N_2,\dots\},\quad
\ext X_j=\ext Y_j=\cm_0(R).
$$
Put $\p=(y)\in\spec R$.
It is clear that $(M_0)_\p=B_\p=0\ne A_\p$.
For each $j\in\Z_{>0}$, as $M_j\cong(xy,y^{j+1})R$ and $Y_j\cong(x,y^j)R$, we have $(M_j)_\p\subseteq\p R_\p=0$ and $(Y_j)_\p\cong R_\p$.
There is an exact sequence $0\to X_j\to R^{\oplus2}\to Y_j\to0$, which shows that $(X_j)_\p\cong R_\p$.
The exact sequence $\sigma_j$ shows $(N_j)_\p\ne0$.
Hence there is an equality
$$
\{E\in\ind\cm_0(R)\mid E_\p=0\}=\{B,M_j\mid j\in\Z_{>0}\}.
$$
Since every $E\in\ext M_j$ is such that $E_\p=0$, we see that $\ext M_j=\add\{B,M_1,M_2,\dots\}$.
Lemma \ref{4}(1) yields that $\ext N_j=\add\{R,A,N_1,N_2,\dots\}$ for every $j\in\Z_{>0}$.
Lemma \ref{3}(3) gives rise to an equality $\ext\{R,M_j\}=\add\{R,B,M_1,M_2,\dots\}$ for every $j\in\Z_{>0}$.
The exact sequence $\sigma_1$ shows that $X_2\in\ext\{M_1,N_1\}$, which implies that $\ext\{M_1,N_1\}=\cm_0(R)$.
Therefore, $\ext\{M_j,N_h\}=\cm_0(R)$ for all $j,h\in\Z_{>0}$.
From $\sigma_0$ and the proof of Lemma \ref{3}(3) we get an exact sequence $0\to A\to X_1\to B\to0$, which implies $X_1\in\ext\{A,B\}$, and $\ext\{A,B\}=\cm_0(R)$.
Now we obtain the Hasse diagram as in the theorem.
\end{proof}

We close the section by posing a question.
As a consequence of Theorems \ref{22}, \ref{23}, \ref{24}, \ref{21} and \ref{26}, this question has an affirmative answer provided that $R$ is a $P$-singularity over an algebraically closed uncountable field of characteristic zero, where $P\in\{\a_n,\a_\infty,\d_n,\d_\infty,\e_6,\e_7,\e_8\}$.
Note that this question is closely related to the problem studied intensively in \cite{tes} in the case of artinian local rings.

\begin{ques}
Let $R=k[\![x_0,x_1,\dots,x_d]\!]/(f)$ be a complete local hypersurface domain over a field $k$ with dimension at most two.
Then does it hold that $\E_0(R)=\{0,\add R,\cm_0(R)\}$\,?
\end{ques}

\section{Proofs of the main theorems}

Now we have reached the stage to prove Theorem \ref{28}, which are the main results of this paper.

\begin{proof}[Proofs of Theorem \ref{28}]
Since $R$ has finite or countable CM-representation type, it is isomorphic to one of the following hypersurfaces; see \cite[Theorems 9.8 and 14.16]{LW}.
$$
\begin{cases}
\a_n^d:\ k[\![x_0,x_1,\dots,x_d]\!]/(x_0^{n+1}+x_1^2+x_2^2+\cdots+x_d^2)&(n\ge1),\\
\a_\infty^d:\ k[\![x_0,x_1,\dots,x_d]\!]/(x_1^2+x_2^2+\cdots+x_d^2),\\
\d_n^d:\ k[\![x_0,x_1,\dots,x_d]\!]/(x_0^{n-1}+x_0x_1^2+x_2^2+\cdots+x_d^2)&(n\ge4),\\
\d_\infty^d:\ k[\![x_0,x_1,\dots,x_d]\!]/(x_0x_1^2+x_2^2+\cdots+x_d^2),\\
\e_6^d:\ k[\![x_0,x_1,\dots,x_d]\!]/(x_0^4+x_1^3+x_2^2+\cdots+x_d^2),\\
\e_7^d:\ k[\![x_0,x_1,\dots,x_d]\!]/(x_0^3x_1+x_1^3+x_2^2+\cdots+z_d^2),\\
\e_8^d:\ k[\![x_0,x_1,\dots,x_d]\!]/(x_0^5+x_1^3+x_2^2+\cdots+z_d^2).
\end{cases}
$$
Thus Theorem \ref{28}(2) is a direct consequence of Theorems \ref{22}, \ref{23}, \ref{24}, \ref{25}, \ref{17}, \ref{21}, \ref{26}, \ref{30} and \ref{27}.
To prove Theorem \ref{28}(1), Kn\"orrer's periodicity theorem \cite[Theorem (12.10)]{Y} reduces to the case where $\dim R\le2$.
The triangle equivalence $\lcm(R)\cong\ds(R)$ induces a triangle equivalence $\lcm_0(R)\cong\dso(R)$.
Hence the Hasse diagram of the extension-closed subcategories of $\dso(R)$ are the same as the Hasse diagram of the extension-closed subcategories of $\lcm_0(R)$, which is obtained from the Hasse diagrams given in Theorems \ref{22}, \ref{23}, \ref{24}, \ref{25}, \ref{17}, \ref{21}, \ref{26}, \ref{30} and \ref{27} by removing the vertices containing $R$ and the edges from/to them.
\end{proof}



\end{document}